\documentclass[final,5p]{elsarticle}

\usepackage{amsmath,amsthm,amssymb}
\usepackage{graphicx}
\usepackage[dvips]{color}
\definecolor{gray}{rgb}{0.5,0.5,0.5}

\def\H{\widetilde H}
\def\N{\mathbb N}
\def\P{\mathbb P}
\def\R{\mathbb R}

\def\x{\boldsymbol{x}}

\def\AA{\mathcal A}

\def\KK{\mathcal K}
\def\MM{\mathcal M}
\def\MM{\mathcal M}
\def\OO{\mathcal O}
\def\PP{\mathcal P}
\def\SS{\mathcal S}
\def\TT{\mathcal T}

\def\uu{\mathfrak u}
\def\vv{\mathfrak v}
\def\ff{\mathfrak f}
\def\gg{\mathfrak g}

\def\Csat{C_{\rm sat}}
\def\diam{{\rm diam}}
\def\length{{\rm length}}

\def\norm#1#2{\|#1\|_{#2}}
\def\set#1#2{\big\{#1\,:\,#2\big\}}
\def\enorm#1{|\!|\!|#1|\!|\!|}
\def\dual#1#2#3{\langle#1\,,\,#2\rangle_{#3}}

\def\qed{\hfill\mbox{$\blacksquare$}}
\def\eps{\varepsilon}
\def\next{\noindent\mbox{\color{gray}\Large$\bullet$}\hspace*{1ex}}

\newtheorem{proposition}{Proposition}

\newtheorem{theorem}[proposition]{Theorem}
\newtheorem{lemma}[proposition]{Lemma}
\newtheorem{remark}[proposition]{Remark}

\newcounter{const}
\def\c#1{C_{\ref{const:#1}}}
\def\newc#1{\refstepcounter{const}\label{const:#1}C_{\theconst}}

\def\revision#1{#1}
\def\mf#1{#1}
\def\subsection#1{\medskip\noindent\refstepcounter{subsection}{\em\bfseries\thesubsection.\ #1}\\[-2mm]}

\begin{document}

\begin{frontmatter}
\title{ZZ-Type A~Posteriori Error Estimators\\for Adaptive Boundary Element Methods on a Curve}
\date{\today}

\author{Michael Feischl}
\ead{Michael.Feischl@tuwien.ac.at}

\author{Thomas F\"uhrer}
\ead{Thomas.Fuehrer@tuwien.ac.at}

\author{Michael Karkulik}
\ead{mkarkulik@mat.puc.cl}

\author{Dirk Praetorius\corref{cor1}}
\ead{Dirk.Praetorius@tuwien.ac.at}
\ead[url]{http://www.asc.tuwien.ac.at/abem}

\cortext[cor1]{Corresponding Author}

\begin{abstract}
In the context of the adaptive finite element method (FEM), ZZ-error estimators
named after Zienkiewicz and Zhu~\cite{zz} are mathematically well-established
and widely used in practice. In this work, we propose and analyze ZZ-type
error estimators for the adaptive boundary element method (BEM). We consider
weakly-singular and hyper-singular integral equations and prove, in particular,
convergence of the related adaptive mesh-refining algorithms. 
Throughout, the theoretical findings are underlined by numerical experiments.
\end{abstract}


\end{frontmatter}

\section{Introduction}

\noindent
Since the seminal work of Zienkiewicz and Zhu~\cite{zz}, averaging techniques became 
popular in engineering and applied sciences for the a~posteriori error control
of the finite element solution of partial differential equations. To sketch the 
idea, we consider  the most simple context of the 2D Poisson equation
\begin{align}
\begin{split}
 -\Delta \uu &= \ff \hspace*{4.3mm}\text{in }\Omega,\\
 \uu &= 0\quad\text{on }\partial\Omega.
\end{split}
\end{align}
Here and throughout the work, $\Omega\subset\R^2$ is a bounded Lip\-schitz domain
with polygonal boundary $\partial\Omega$.

Let $\TT_h$ denote a regular triangulation of $\Omega$ into compact,
nondegenerate triangles. Let $\PP^0(\TT_h)$ be the space of all $\TT_h$-piecewise
constant functions and $\SS^1(\TT_h)$ be the space of all $\TT_h$-piecewise affine
and globally continuous splines. The lowest-order finite element solution
$\uu_h\in\SS^1_0(\TT_h) := \set{\vv_h\in\SS^1(\TT_h)}{\vv_h=0\text{ on }\partial\Omega}$ is the unique solution of the Galerkin formulation
\begin{align}
 \int_\Omega \nabla \uu_h\cdot\nabla \vv_h\,dx
 = \int_\Omega \ff \vv_h\,dx
\end{align}
for all test functions $\vv_h\in\SS^1_0(\TT_h)$. In this context, the ZZ error 
estimator reads
\begin{align}
 \eta_h = \norm{(1-\AA_h)\nabla \uu_h}{L^2(\Omega)},
\end{align}
where $\AA_h:\PP^0(\TT_h)^2\to\SS^1(\TT_h)^2$ is some averaging operator which 
maps the $\TT_h$-piecewise constant gradient $\nabla \uu_h\in\PP^0(\TT_h)^2$ onto 
some continuous and piecewise affine function 
$\AA_h\nabla \uu_h\in\SS^1(\TT_h)^2$. Possible choices for $\AA_h$ are the 
usual Cl\'ement-type operators like
\begin{align}
 (\AA_h \vv)(z) = \frac{1}{{\rm area}(\omega_z)}\,\int_{\omega_z}\vv\,dx
\end{align}
for all nodes $z\in\KK_h$ of $\TT_h$, where 
\begin{align}
 \omega_z:=\bigcup\set{T\in\TT_h}{z\in T}
\end{align}
denotes the patch of $z$, i.e., the union of all elements $T\in\TT_h$ which have
$z$ as a node.
Although ZZ error estimators are strikingly simple and mathematically well-developed
for the finite element method, see e.g.~\cite{bc1,bc2,cc04,r94}, they have not been considered
for boundary element methods, yet. 
\revision{Available error estimators from the 
literature include residual-based error estimators for weakly-singular
\cite{cs95,cceps96,cc97,cf,cms,faermann2d,faermann3d} and hyper-singular integr\mf{a}l equations~\cite{cc97,cmps}, hierarchical error estimators for weakly-singular
\cite{eh06,hms,msw} and hyper-singular integral 
equations \cite{h02,hms}, $(h-h/2)$-based error 
estimators \cite{effp,efgp,fp}, averaging on large 
patches~\cite{cp:symm,cp:hypsing,cp:steinbach},
and estimators based on the use of the full Calder\'on system~\cite{paulino,ss00,steinbach00}.
 The reader is also
referred to the overviews given in~\cite{cf,efgp} and the references therein.}

\revision{This note proposes} ZZ-type error estimators in the context of the 
boundary element method. As model problems serve the hyper-singular and the weakly-singular integral equation associated with the 2D 
Laplacian. Difficulties arise from the fact that neither the involved integral 
operators nor the energy norms are local. 
%

The outline of this paper reads as follows: In Section~\ref{section:hypsing},
we consider the hyper-singular integral equation, introduce a ZZ-type error
estimator, and provide numerical evidence for its successful use on a slit model
problem as well as for the first-kind integral formulation of some Neumann 
problem. In Section~\ref{section:symm}, \revision{we apply this approach} in the context of the
weakly-singular integral equation. 
While Section~\ref{section:hypsing} and Section~\ref{section:symm} are 
written  for a general audience, Section~\ref{section:preliminaries} collects
the preliminaries for the numerical analysis of the proposed a~posteriori error estimators. 
A rigorous a~posteriori error analysis
is postponed to Section~\ref{section:aposteriori}.
The final Section~\ref{section:convergence} even proves convergence of the 
standard adaptive mesh-refining algorithm steered by the ZZ-type error 
estimators proposed.

\section{Hyper-singular integral equation}
\label{section:hypsing}

\noindent
We \revision{suppose that $\Omega\subset\R^2$ is} simply connected, i.e., $\Omega$ has no
holes and $\partial\Omega$ thus is connected.
We denote the fundamental solution of the 2D Laplacian by
\begin{align}\label{eq:kernel}
 G(z) := -\frac{1}{2\pi}\,\log|z|
 \quad\text{for }z\in\R^2\backslash\{0\}.
\end{align}
Let $\Gamma$ be some relatively open and connected subset of the boundary 
$\partial\Omega$. Then, the hyper-singular integral operator is formally 
defined by
\begin{align}
 (Wu)(x) = -\partial_{n(x)}\int_\Gamma\partial_{n(y)}G(x-y)u(y)\,d\Gamma(y)
\end{align}
for $x\in\Gamma$.
Here, $\int_\Gamma\,d\Gamma$ denotes integration along the curve and 
$\partial_{n(x)}$ is the normal derivative at some point $x\in\Gamma$. 
The hyper-singular integral equation reads
\begin{align}\label{eq:hypsing}
 Wu = f
 \quad\text{on }\Gamma.
\end{align}
For the following facts on the functional analytic setting as well as for
proofs and further details, the reader is referred to e.g.\ the 
monographs~\cite{hw,mclean,ss}.

\subsection{Slit model problem}

\noindent
Assume that $\Gamma\subsetneqq\partial\Omega$ is not closed. 
\revision{Let $\H^{1/2}(\Gamma)$ denote the space of all 
$H^{1/2}(\Gamma)$-functions which vanish at the tips of $\Gamma$.}
Then, 
$W$ is a linear, bounded and elliptic operator between the fractional-order 
Sobolev space $\H^{1/2}(\Gamma)$ and its dual space $H^{-1/2}(\Gamma)$, where 
duality is understood with respect to the extended $L^2(\Gamma)$-scalar
product $\dual\cdot\cdot{L^2(\Gamma)}$. Let $f\in H^{-1/2}(\Gamma)$. 
The variational form \revision{of~\eqref{eq:hypsing} reads}
\begin{align}\label{eq:hypsing:slit}
 \dual{Wu}{v}{L^2(\Gamma)} = \dual{f}{v}{L^2(\Gamma)}
 \quad\text{for all }v\in\H^{1/2}(\Gamma).
\end{align}
Since the left-hand side defines a scalar product on $\H^{1/2}(\Gamma)$, 
the Lax-Milgram lemma provides existence and uniqueness of the solution $u$.

\subsection{Model problem on closed boundaries}

\noindent
Assume that $\Gamma=\partial\Omega$ is closed. Then, $W$ is a linear and bounded
operator from $H^{1/2}(\partial\Omega)$ to $H^{-1/2}_\star(\partial\Omega)
:=\set{\psi\in H^{-1/2}(\partial\Omega)}{\dual{\psi}{1}{L^2(\revision{\partial\Omega})}=0}$. Moreover, $W$ is 
elliptic on the subspace $H^{1/2}(\partial\Omega)/\R\equiv H^{1/2}_\star(\partial\Omega) 
:= \set{v\in H^{1/2}(\partial\Omega)}{\int_{\partial\Omega} v\,d\Gamma=0}$, where connectedness
of $\partial\Omega$ is required. Let $f\in H^{-1/2}_\star(\partial\Omega)$. 
The variational form \revision{of~\eqref{eq:hypsing} now reads}
\begin{align}\label{eq:hyping:closed}
 \dual{Wu}{v}{L^2(\partial\Omega)} = \dual{f}{v}{L^2(\partial\Omega)}
 \text{ for all }v\in H^{1/2}_\star(\partial\Omega).
\end{align}
As before, the left-hand side defines a scalar product on $H^{1/2}_\star(\partial\Omega)$, 
and the Lax-Milgram lemma thus provides existence and uniqueness of the solution $u$.

We note that, for certain right-hand sides $f$ and $\Gamma=\partial\Omega$,
\eqref{eq:hypsing} is an equivalent formulation of the Neumann problem
\begin{align}\label{eq:hypsing:pde}
\begin{split}
 -\Delta \uu &= \ff \quad\text{in }\Omega,\\
 \partial_n \uu &= \gg\quad\text{on }\partial\Omega.
\end{split}
\end{align}
In this case, the solution $u$ of~\eqref{eq:hypsing} is, up to some additive
constant, the trace $\uu|_{\partial\Omega}$ of the solution $\uu$ 
of~\eqref{eq:hypsing:pde}.

\subsection{Galerkin boundary element discretization}

\noindent
Let $\TT_h$ be a partition of $\Gamma$ into affine line segments.
Let $\SS^1(\TT_h)$ denote the space of all functions $v_h$ which are 
continuous and $\TT_h$-piecewise affine with respect to the arc\-length.
For $\Gamma\subsetneqq\partial\Omega$, $\SS^1_0(\TT_h)
:=\SS^1(\TT_h)\cap\H^{1/2}(\Gamma)$ 
denotes the subspace of all functions $v_h\in\SS^1(\TT_h)$ 
which additionally vanish at the tips of $\Gamma$. 
For $\Gamma = \partial\Omega$, $\SS^1_0(\TT_h)
:=\SS^1(\TT_h)\cap H^{1/2}_\star(\Gamma)$ 
denotes the subspace of all functions $v_h\in\SS^1(\TT_h)$ which satisfy 
$\int_\Gamma v_h\,d\Gamma=0$. In either case, $\SS^1_0(\TT_h)$ 
is a conforming subspace of $\H^{1/2}(\Gamma)$ resp.\ $H^{1/2}_\star(\partial\Omega)$. In particular, the Galerkin formulation
\revision{of~\eqref{eq:hypsing:slit} resp.~\eqref{eq:hyping:closed} reads}
\begin{align}
 \dual{Wu_h}{v_h}{L^2(\Gamma)} = \dual{f}{v_h}{L^2(\Gamma)}
 \text{ for all }v_h\in\SS^1_0(\TT_h)
\end{align}
\revision{and} admits a unique Galerkin solution $u_h\in\SS^1_0(\TT_h)$. 

\subsection{ZZ-type error estimator}
\label{section:hypsing:zz}

\noindent
Let $h\in L^\infty(\Gamma)$ be the local mesh-size function defined by 
\begin{align}\label{eq:h}
 h|_T:=\length(T) 
 \quad\text{for } 
 T\in\TT_h
\end{align}
with the arclength $\length(\cdot)$.
With $(\cdot)'$ denoting the arc\-length 
derivative, we propose the following ZZ-type 
error estimator
\begin{align}\label{eq1:hypsing:clement}
 \eta_h = \norm{h^{1/2}(1-\AA_h)u_h'}{L^2(\Gamma)},
\end{align}
where $\AA_h:\revision{L^2(\Gamma)}\to\SS^1(\TT_h)$ denotes the Cl\'ement operator 
defined by
\begin{align}\label{eq2:hypsing:clement}
 (\AA_hv)(z) 
 := \frac{1}{\length(\omega_z)}\int_{\omega_z}v\,d\Gamma
\end{align}
for all nodes $z\in\KK_h\text{ of }\TT_h$
with $\omega_z = \bigcup\set{T\in\TT_h}{z\in T}$ the node patch.

\subsection{Adaptive mesh-refining algorithm}
\label{section:hypsing:algorithm}

\noindent
\revision{Given} a right-hand side $f\in H^{-1/2}(\Gamma)$, an initial
partition $\TT_h$ of $\Gamma$, and some adaptivity parameter $0<\theta<1$,
\revision{the proposed} adaptive algorithm reads as follows:
\begin{itemize}
\item[(i)] Compute discrete solution $u_h\in\SS^1_0(\TT_h)$.
\item[(ii)] For all $T\in\TT_h$, compute the refinement \revision{indicators}
\begin{align}
 \eta_h(T)^2 := \length(T)\,\norm{(1-\AA_h)u_h'}{L^2(T)}^2.
\end{align}
\item[(iii)] Determine a set $\MM_h\subseteq\TT_h$ such that
\begin{align}\label{eq:doerfler}
 \theta\,\eta_h^2 \le \sum_{T\in\MM_h}\eta_h(T)^2.
\end{align}
\item[(iv)] Generate a new mesh $\TT_h$ by bisection of at least all elements 
in $\MM_h$.
\item[(v)] goto (i) and iterate.
\end{itemize}
\revision{%
For the proof of quasi-optimal convergence rates in the frame of adaptive 
FEM, e.g.~\cite{s07,ckns}, and adaptive BEM~\cite{fkmp,gantumur}, 
the set $\MM_h$ in step~(iii) is usually chosen with minimal cardinality.
A greedy algorithms sorts the indicators in descending order and then iteratively splits $\TT_h$ into $\MM_h$ and $\TT_h\backslash\MM_h$ by moving
the largest indicator from $\TT_h\backslash\MM_h$ to $\MM_h$ until the 
D\"orfler criterion~\eqref{eq:doerfler} is satisfied.}

For our implementation, we use the \textsc{Matlab} BEM 
library~\texttt{HILBERT}~\cite{hilbert}. 
The local mesh-refinement in 
step~(iv) of the algorithm is done by some bisection-based algorithm 
from~\cite{affkp} which guarantees that the local mesh-ratio 
\begin{align}\label{eq:kappa}
 \kappa(\TT_h) := \max\set{\frac{\length(T)}{\length(T')}}{T,T'\in\TT_h\text{ neighbors}}
\end{align}
stays uniformly bounded 
$\kappa(\TT_h)\le\gamma$ for some $\gamma\ge2$ which depends only
on the initial partition. 
We stress that such a property is required for the numerical analysis of $\eta_h$ 
in Section~\ref{section:aposteriori} and Section~\ref{section:convergence} below.

\revision{We recall from the literature~\cite{ss} that the optimal rate of 
convergence with lowest-order BEM is $\OO(h^{3/2})$ if the exact solution 
is smooth. This corresponds to $\OO(N^{-3/2})$ with respect to the number 
$N$ of elements on adaptively generated meshes.}

\begin{figure}[t]
\begin{center}
\includegraphics[width=.45\textwidth]{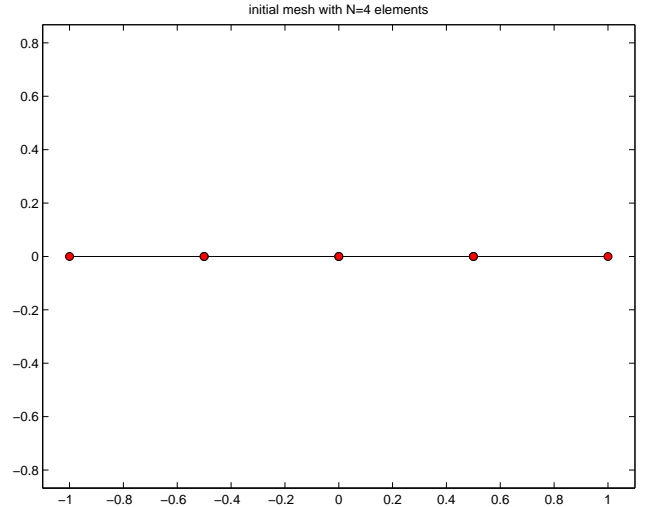}
\caption{Slit $\Gamma=(-1,1)\times\{0\}$ and initial mesh $\TT_h$ with
$N=4$ elements of the numerical experiment for the hyper-singular integral
equation from Section~\ref{example:hypsing:slit} and the weakly-singular
integral equation from Section~\ref{example:weaksing:slit}.}
\label{fig:hypsing:slit}
\end{center}
\end{figure}
\begin{figure}[t]
\begin{center}
\includegraphics[width=.45\textwidth]{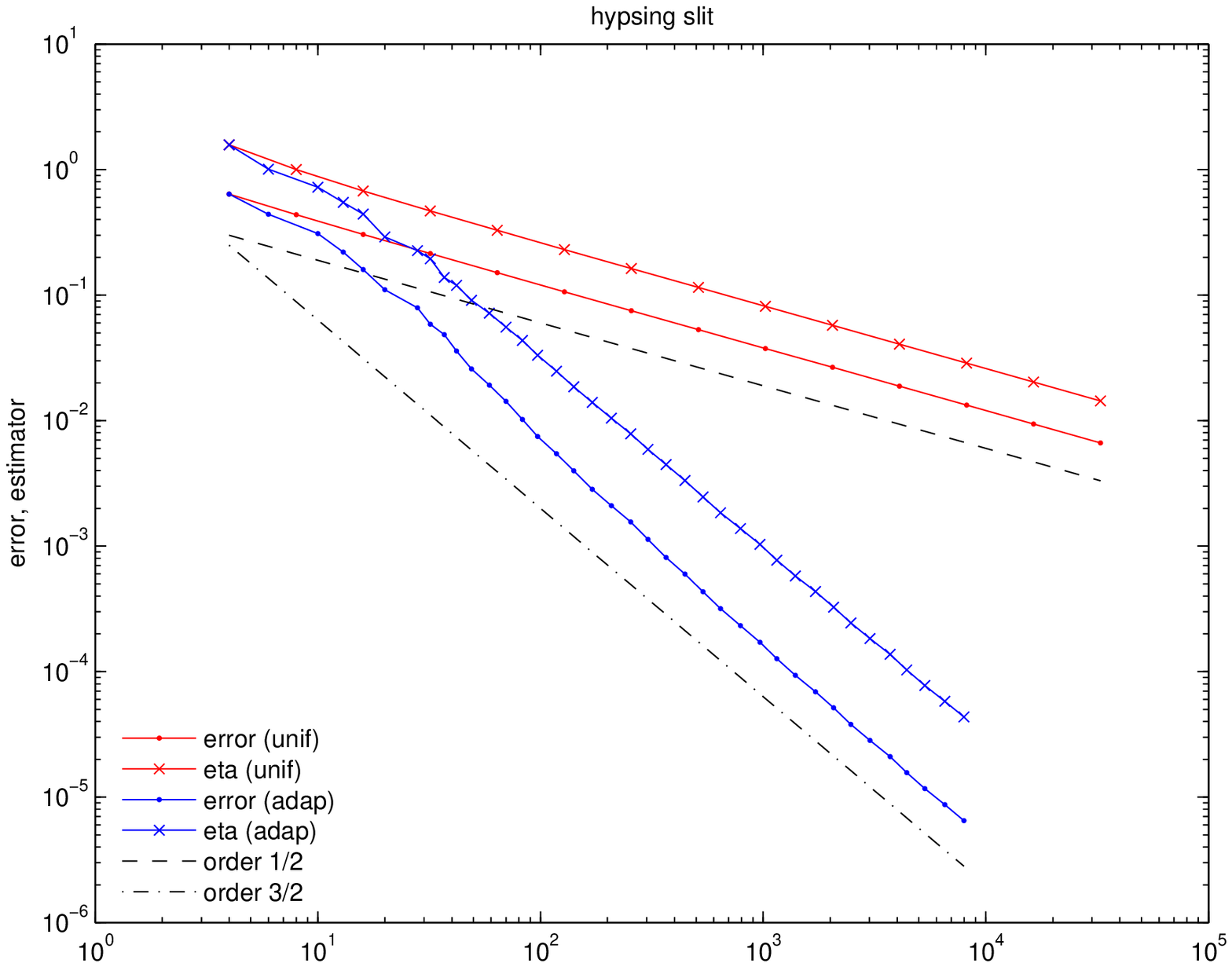}
\caption{Numerical outcome of the experiment for the hyper-singular integral
equation from Section~\ref{example:hypsing:slit}.}
\label{fig:hypsing:slit:results}
\end{center}
\end{figure}
%
\subsection{Numerical experiment for slit problem}
\label{example:hypsing:slit}

\noindent 
We consider the hyper-singular integral equation
\begin{align}
 Wu = 1
 \quad\text{on }\Gamma = (-1,1)\times\{0\}.
\end{align}
The exact solution is known and reads $u(x,0) = 2\sqrt{1-x^2}$. Note that
$u\in\H^{1/2}(\Gamma)\cap H^{1-\eps}(\Gamma)$ for all $\eps>0$. In particular,
we expect \revision{an empirical} convergence order $\OO(h^{1/2})$ for 
uniform mesh-refinement.

The initial mesh $\TT_h$ for the computation
is shown in Figure~\ref{fig:hypsing:slit}. We compare adaptive 
mesh-refinement with parameter $\theta=1/2$ with uniform mesh-refinement.
The corresponding convergence graphs are visualized in 
Figure~\ref{fig:hypsing:slit:results}. \revision{While uniform mesh-refinement
leads to the predicted suboptimal order $\OO(h^{1/2})=\OO(N^{-1/2})$, the
proposed adaptive strategy regains the optimal rate $\OO(N^{3/2})$.}

\begin{figure}[t]
\begin{center}
\includegraphics[width=.45\textwidth]{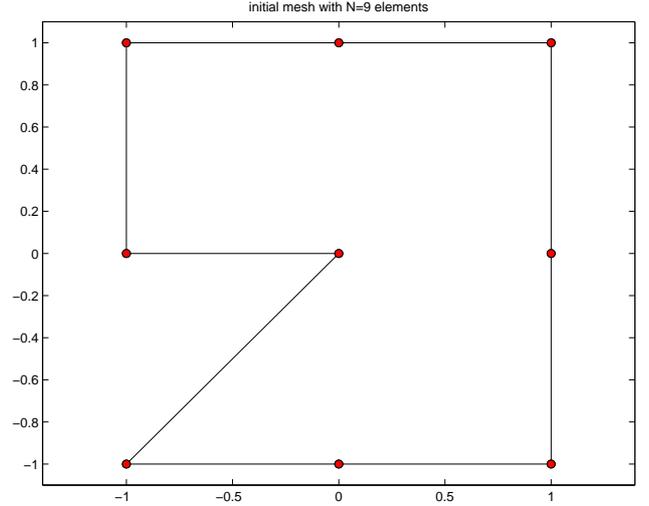}
\caption{Boundary $\Gamma=\partial\Omega$ and initial mesh $\TT_h$ with
$N=9$ elements of the numerical experiment for the hyper-singular integral
equation from Section~\ref{example:hypsing}.}
\label{fig:hypsing:zshape}
\end{center}
\end{figure}
\begin{figure}[t]
\begin{center}
\includegraphics[width=.45\textwidth]{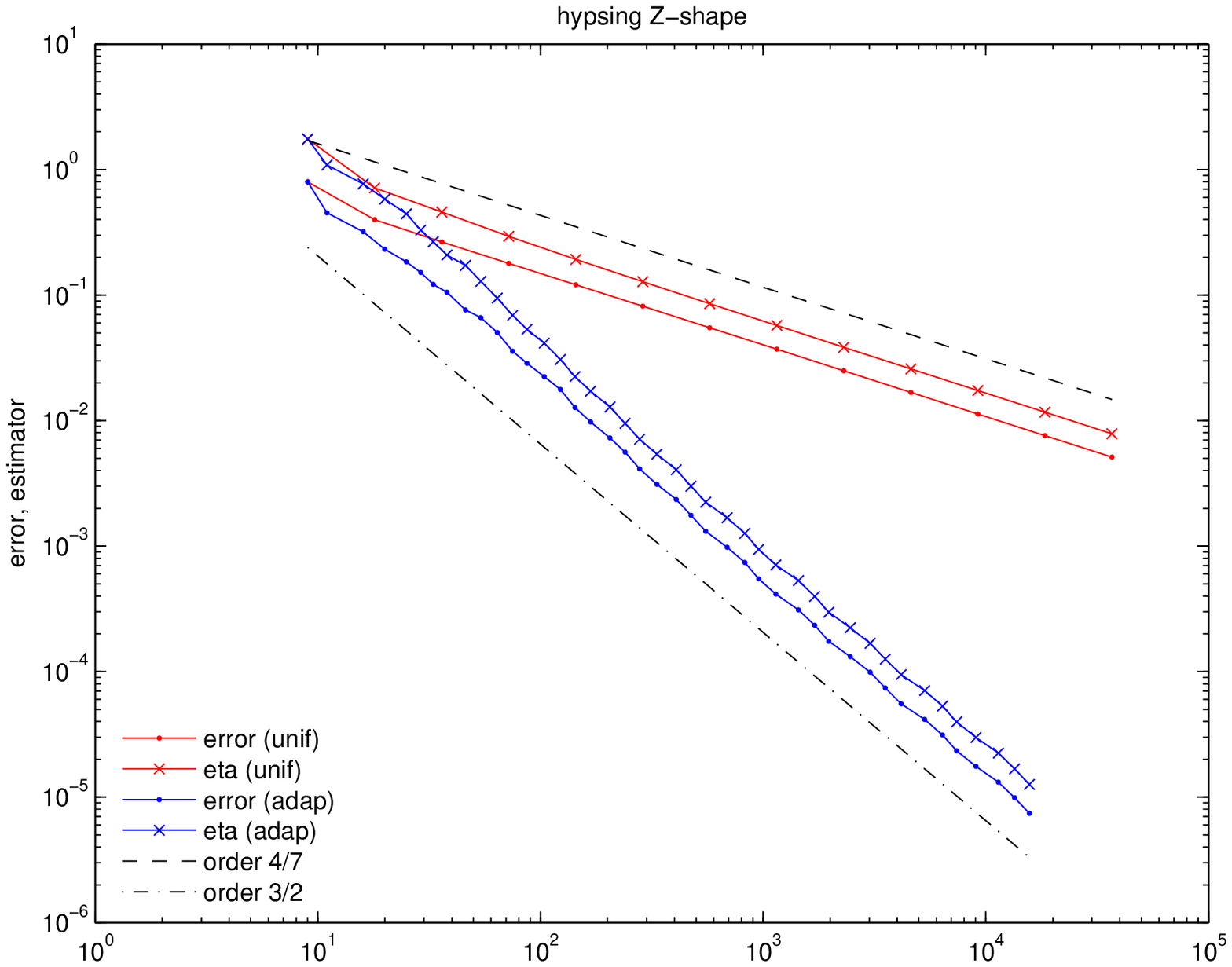}
\caption{Numerical outcome of the experiment for the hyper-singular integral
equation from Section~\ref{example:hypsing}.}
\label{fig:hypsing:zshape:results}
\end{center}
\end{figure}
%
\subsection{Numerical experiment on closed boundary}
\label{example:hypsing}

\noindent
We consider the Z-shaped domain with reentrant corner
at \revision{the origin $(0,0)$}, \revision{see 
Figure~\ref{fig:hypsing:zshape}} \mf{for a sketch}.
The right-hand side \revision{$f=(1/2-K')(\partial_n\uu)\in H^{-1/2}(\Gamma)$ 
with $\Gamma=\partial\Omega$ and $K'$ the adjoint double layer-potential}
is chosen such that the 
hyper-singular integral equation~\eqref{eq:hypsing} is equivalent to some
Neumann problem~\eqref{eq:hypsing:pde} with $\ff=0$. The exact solution reads
\begin{align}
 \uu(x) = r^{4/7}\cos(4\varphi/7)
\end{align}
in 2D polar coordinates $x = r\,(\cos\varphi,\sin\varphi)$. The exact 
solution $u$
of~\eqref{eq:hypsing} is, up to some additive constant, the trace
$\uu|_\Gamma$. Moreover, $u$ admits a generic singularity at the reentrant corner.
\revision{Note that $u\in H^{1/2}_0(\partial\Omega)\cap H^{\mf{4/7+1/2}-\eps}(\partial\Omega)$ for all $\eps>0$. Theoretically, this predicts an expected convergence
order $\OO(h^{4/7})$ for uniform mesh-refinement.}

The Z-shaped domain as well as the initial mesh $\TT_h$ for the computation
are shown in Figure~\ref{fig:hypsing:zshape}. We compare adaptive 
mesh-refinement with parameter $\theta=1/2$ with uniform mesh-refinement.
The corresponding convergence graphs are visualized in 
Figure~\ref{fig:hypsing:zshape:results}. 
\revision{While uniform mesh-refinement leads \revision{to the expected rate}
$\OO(h^{4/7})=\OO(N^{-4/7})$,
the proposed adaptive strategy regains the optimal rate $\OO(N^{-3/2})$.}

\section{Weakly-singular integral equation}
\label{section:symm}

\noindent
In this section, we consider the simple-layer potential
\begin{align}
 (V\phi)(x) = \int_\Gamma G(x-y)\phi(y)\,d\Gamma(y)
 \quad\text{for }x\in\Gamma,
\end{align}
where $G(\cdot)$ denotes the fundamental solution of the 2D Laplacian 
from~\eqref{eq:kernel}. We assume that $\Gamma\subseteq\partial\Omega$ is 
a relatively open but possibly non-connected subset of the boundary $\partial\Omega$ and that $\diam(\Omega)<1$. 
For the following facts on the functional analytic setting as well as for
proofs and further details, we again refer to e.g.\ the 
monographs~\cite{hw,mclean,ss}.

\subsection{Model problem}

\noindent
It is well-known that $V$ is a linear, bounded, and elliptic operator
from $\H^{-1/2}(\Gamma)$ to its dual $H^{1/2}(\Gamma)$, 
\revision{where ellipticity follows from $\diam(\Omega)<1$}.
Given some $f\in H^{1/2}(\Gamma)$, we aim at the numerical solution of the 
weakly-singular integral equation
\begin{align}\label{eq:weaksing}
 V\phi = f.
\end{align}
We use the variational form
\begin{align}\label{eq:weaksing:weakform}
 \dual{V\phi}{\psi}{L^2(\Gamma)} = \dual{f}{\psi}{L^2(\Gamma)}
 \quad\text{for all }\psi\in\H^{-1/2}(\Gamma).
\end{align}
The left-hand side defines an equivalent scalar product on 
$\H^{-1/2}(\Gamma)$, and the Lax-Milgram lemma thus provides existence and
uniqueness of the solution $\phi\in\H^{-1/2}(\Gamma)$ 
of~\eqref{eq:weaksing:weakform}. 

We stress that, for certain right-hand sides $f$ and 
$\Gamma=\partial\Omega$,~\eqref{eq:weaksing} is an equivalent formulation 
of the Dirichlet problem
\begin{align}\label{eq:weaksing:pde}
\begin{split}
 -\Delta \uu &= \ff \quad\text{in }\Omega,\\
 \uu &= \gg\quad\text{on }\Gamma.
\end{split}
\end{align}
In this case, it holds $\phi=\partial_n\uu$. In particular, one cannot 
expect that $\phi$ is locally smooth, where the outer normal vector $n$ 
is not.

\subsection{Galerkin boundary element discretization}

\noindent
Let $\TT_h$ be a partition of $\Gamma$ into affine line segments.
Let $\PP^0(\TT_h)$ denote the space of all $\TT_h$-piecewise constant
functions $\psi_h$.
For the Galerkin discretization, we
replace $\phi,\psi\in\H^{-1/2}(\Gamma)$ by discrete functions 
$\phi_h,\psi_h\in\PP^0(\TT_h)$. Then, $\PP^0(\TT_h)\subset\H^{-1/2}(\Gamma)$ 
is a conforming subspace, and the Galerkin formulation
\begin{align}
 \dual{V\phi_h}{\psi_h}{L^2(\Gamma)} = \dual{f}{\psi_h}{L^2(\Gamma)}
 \text{ for all }\psi_h\in\PP^0(\TT_h)
\end{align}
admits a unique Galerkin solution $\phi_h\in\PP^0(\TT_h)$.

\subsection{ZZ-type error estimator}

\noindent
With $h\in L^\infty(\Gamma)$ the local mesh-size function from~\eqref{eq:h}, 
we propose the following ZZ-type error estimator
\begin{align}
 \eta_h = \norm{h^{1/2}(1-\AA_h)\phi_h}{L^2(\Gamma)}.
\end{align}
As noted before, we may expect that $\phi$ is non-smooth at points 
$x\in\Gamma$, where the normal mapping $x\mapsto n(x)$ is non-smooth. 
Therefore, we slightly modify the Cl\'ement operator 
$\AA_h:L^2(\Gamma)\to\PP^1(\TT_h)$ 
\revision{from~\eqref{eq2:hypsing:clement}}
as follows:
\begin{itemize}
\item First, if $\{z\}=T_j\cap T_k$ is the node between the elements 
$T_j,T_k\in\TT_h$ and if the normal vector of $T_j$ and $T_k$ does not jump 
at $z$, we define
\begin{align}\label{eq1:weaksing:clement}
 (\AA_hv)(z)
 := \frac{1}{\length(\omega_z)}\int_{\omega_z}v\,d\Gamma
\end{align}
with $\omega_z = \bigcup\set{T\in\TT_h}{z\in T} = T_j\cup T_k$ the node patch. 
\item Second, if the normal vectors of $T_j$ and $T_k$ differ at 
$z$, we allow $\AA_hv$ to jump at $z$ as well, namely
\begin{align}\label{eq2:weaksing:clement}
\begin{split}
(\AA_hv)|_{T_j}(z) &= \frac{1}{\length(T_j)}\int_{T_j}v\,d\Gamma,\\
(\AA_hv)|_{T_k}(z) &= \frac{1}{\length(T_k)}\int_{T_k}v\,d\Gamma.
\end{split}
\end{align}
\end{itemize}
Note that this definition can only be meaningful if each \revision{connected}
component $\gamma\subseteq\Gamma$ on which the normal mapping $x\mapsto n(x)$
is smooth, consists \revision{of at least} two elements. Otherwise, $\gamma=T_j$ would 
lead to $\phi_h|_\gamma = (\AA_h\phi_h)|_\gamma$ so that $\eta_h$ vanishes
on $\gamma$, i.e.\ $T_j$ would never be marked for refinement by an adaptive
algorithm.

\subsection{Adaptive algorithm}
\label{section:weaksing:algorithm}

\noindent
We consider the adaptive algorithm from Section~\ref{section:hypsing:algorithm}
with the obvious modifications., i.e.\ we compute $\phi_h\in\PP^0(\TT_h)$ in
step~(i) as well as the local contributions
\begin{align}
 \eta_h(T)^2 := \length(T)\,\norm{(1-\AA_h)\phi_h}{L^2(T)}^2
\end{align}
in step~(ii). \revision{We refer to the literature, e.g.~\cite{ss}, that the optimal
rate of lowest-order BEM is $\OO(h^{3/2})$ for a smooth solution $\phi$,
and the adaptive algorithm thus aims to regain a convergence order
$\OO(N^{-3/2})$ with respect to the number of elements.}

%
\begin{figure}[t]
\begin{center}
\includegraphics[width=.45\textwidth]{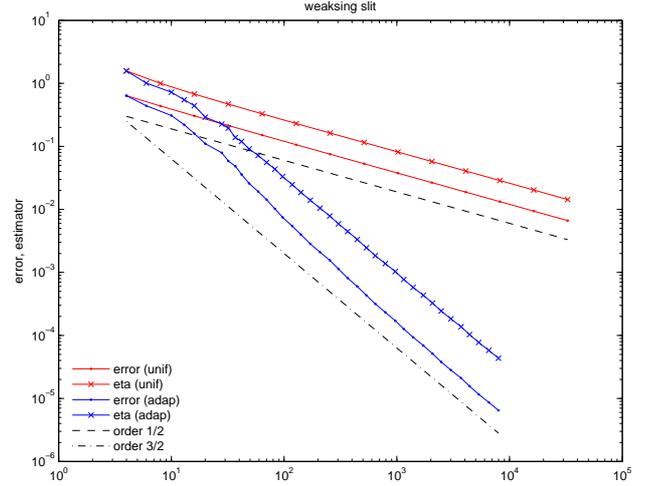}
\caption{Numerical outcome of experiment from Section~\ref{example:weaksing:slit}.}
\label{fig:weaksing:slit:results}
\end{center}
\end{figure}
%
\subsection{Numerical experiment for slit problem}
\label{example:weaksing:slit}

\noindent 
We consider the weakly-singular integral equation
\begin{align}
 V\phi = 1
 \quad\text{on }\Gamma = (-1,1)\times\{0\}.
\end{align}
The unique exact solution \revision{of this equation} is known and reads $\phi(x,0) = -2x/\sqrt{1-x^2}$. Note that
$\phi\in\H^{-1/2}(\Gamma)\cap H^{-\eps}(\Gamma)$ for all $\eps>0$. In 
particular, we expect \revision{an empirical convergence} order $\OO(h^{1/2})$ for uniform 
mesh-refinement.

The initial mesh $\TT_h$ for the computation
is shown in Figure~\ref{fig:hypsing:slit}. We compare adaptive 
mesh-refinement with parameter $\theta=1/2$ with uniform mesh-refinement.
The corresponding convergence graphs are visualized in 
Figure~\ref{fig:weaksing:slit:results}.
\revision{While uniform mesh-refinement leads to the expected rate 
$\OO(h^{1/2})=\OO(N^{-1/2})$, the adaptive algorithm regains the optimal
rate $\OO(N^{-3/2})$.}

\begin{figure}[t]
\begin{center}
\includegraphics[width=.45\textwidth]{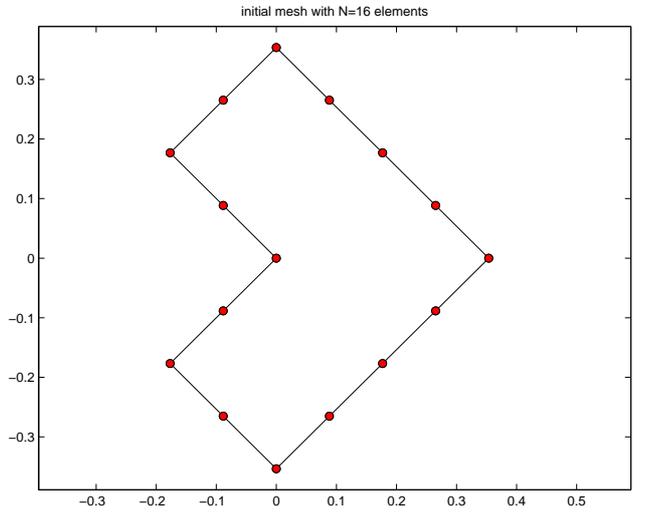}
\caption{Boundary $\Gamma=\partial\Omega$ and initial mesh $\TT_h$ with
$N=16$ elements of the numerical experiment from 
Section~\ref{example:weaksing}.}
\label{fig:weaksing:lshape}
\end{center}
\end{figure}
\begin{figure}[t]
\begin{center}
\includegraphics[width=.45\textwidth]{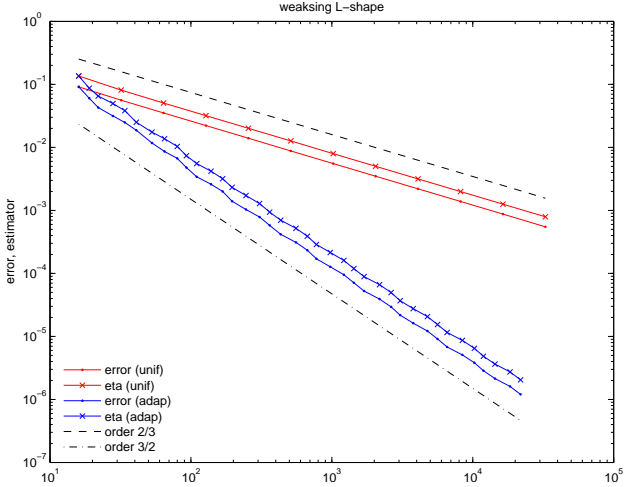}
\caption{Numerical outcome of experiment from Section~\ref{example:weaksing}.}
\label{fig:weaksing:lshape:results}
\end{center}
\end{figure}
%
\subsection{Numerical experiment on closed boundary}
\label{example:weaksing}

\noindent
We consider the rotated L-shaped domain \revision{from Figure~\ref{fig:weaksing:lshape}} with reentrant 
corner \revision{at the origin $(0,0)$}. 
We consider
$\Gamma=\partial\Omega$ and choose the right-hand side 
\revision{$f=(K+1/2)(\uu|_{\Gamma})\in H^{1/2}(\Gamma)$
with $K$ the double-layer potential, so}
that the weakly-singular integral equation~\eqref{eq:weaksing} is 
equivalent to some Dirichlet problem~\eqref{eq:weaksing:pde} with 
\revision{$\ff=0$}.
The exact solution \revision{of~\eqref{eq:weaksing:pde}} is prescribed as
\begin{align}
\uu(x) = r^{2/3}\cos(2\varphi/3)
\end{align}
in 2D polar coordinates $x=r\,(\cos\varphi,\sin\varphi)$
and admits a generic singularity at the reentrant corner.
The exact solution $\phi$ of~\eqref{eq:weaksing} is the normal derivative
$\phi=\partial_n\uu$.
\revision{We note that $\phi\in H^{\mf{2/3-1/2}-\eps}(\Gamma)$ for all $\eps>0$, and
we may hence expect convergence of order $\OO(h^{2/3})$ for uniform 
mesh-refinement.}

The L-shaped domain as well as the initial mesh $\TT_h$ for the computation
are shown in Figure~\ref{fig:weaksing:lshape}. We compare adaptive 
mesh-refinement with parameter $\theta=1/2$ with uniform mesh-refinement.
The corresponding convergence graphs are visualized in 
Figure~\ref{fig:weaksing:lshape:results}.
\revision{The proposed adaptive algorithm recovers the optimal order of
convergence.}

\section{Preliminaries}
\label{section:preliminaries}

\noindent
The purpose of this short section is to fix the notation of the spaces
involved and to recall standard results used in the following.

\subsection{Interpolation spaces}
\label{section:interpolation}

\noindent
Let $X_0$ and $X_1$ be Hilbert spaces with $X_0\supseteq X_1$ and continuous inclusion,
i.e., there exists some constant $C>0$ such that
\begin{align}
 \norm{x}{X_0} \le C\,\norm{x}{X_1}
 \quad\text{for all }x\in X_1.
\end{align}
Interpolation theory, e.g.~\cite{bl76}, provides a means to define
intermediate spaces 
\begin{align}
 X_1 \subseteq X_s := [X_0;X_1]_s\subseteq X_0
 \quad\text{for all }0<s<1,
\end{align}
where $[\cdot;\cdot]_s$ denotes the interpolation operator of, e.g., the
real $K$-method. \revision{The norm related to the intermediate interpolation space $X_s$ satisfies
\begin{align}\label{eq:interp:norm}
 \norm{x}{X_s} \le \norm{x}{X_0}^{1-s}\norm{x}{X_1}^s
 \quad\text{for all }x\in X_1.
\end{align}
The most important consequence, however,} is the so-called \emph{interpolation estimate}:
Let $X_0\supseteq X_1$ and $Y_0\supseteq Y_1$ be Hilbert spaces with continuous
inclusions. Let $T:X_0\to Y_0$ be a linear operator with $T(X_1)\subseteq Y_1$.
Assume that $T:X_0\to Y_0$ as well as $T:X_1\to Y_1$ are continuous, i.e.,
\begin{align}
\begin{split}
 \norm{Tx}{Y_0} &\le \c{norm0}\,\norm{x}{X_0}
 \quad\text{for all }x\in X_0,\\
 \norm{Tx}{Y_1} &\le \c{norm1}\,\norm{x}{X_1}
 \quad\text{for all }x\in X_1,
\end{split}
\end{align}
with the respective operator norms $\newc{norm0},\newc{norm1}>0$. Let $0<s<1$
and $X_s=[X_0;X_1]_s$ and $Y_s=[Y_0;Y_1]_s$. Then, $T:X_s\to Y_s$ is a 
well-defined linear and continuous operator with
\begin{align}\label{eq:interp:operator} 
 \norm{Tx}{Y_s} &\le \c{norm0}^{1-s}\c{norm1}^s\,\norm{x}{X_s}
 \quad\text{for all }x\in X_s.
\end{align}
\revision{Note that for other interpolation methods than the real $K$-method, 
the previous estimates~\eqref{eq:interp:norm} and~\eqref{eq:interp:operator} 
hold only up to some additional generic constants, see e.g.~\cite{bl76}.}

\subsection{Function spaces}
\label{section:spaces}

\noindent
\revision{Let $L^2(\Gamma)$} denote the space of square integrable functions on 
$\Gamma$, associated with the Hilbert norm
\begin{align}
 \norm{v}{L^2(\Gamma)}^2 := \int_\Gamma v^2\,d\Gamma.
\end{align}
Note that $\norm\cdot{L^2(\Gamma)}$ stems from the scalar product
\begin{align}
 \dual{v}{w}{L^2(\Gamma)} := \int_\Gamma vw\,d\Gamma.
\end{align}
\revision{Let $H^1(\Gamma)$} denote the closure of all Lipschitz continuous functions
on $\Gamma$ with respect to the Hilbert norm
\begin{align}
 \norm{v}{H^1(\Gamma)}^2
 := \norm{v}{L^2(\Gamma)}^2 + \norm{v'}{L^2(\Gamma)}^2.
\end{align}
\revision{Let $\H^1(\Gamma)$} denote the closure of all Lipschitz continuous functions
on $\Gamma$ with respect to the $H^1(\Gamma)$-norm which vanish at the tips
of $\Gamma$. We stress that both $H^1(\Gamma)$ and $\H^1(\Gamma)$ are dense
subspaces of $L^2(\Gamma)$ with respect to the $L^2(\Gamma)$-norm. Moreover,
it holds $H^1(\Gamma)=\H^1(\Gamma)$ in case of a closed boundary 
$\Gamma=\partial\Omega$.

Sobolev spaces of fractional order $0<s<1$ are defined by interpolation
\begin{align}
\begin{split}
 H^s(\Gamma) &:= [L^2(\Gamma);H^1(\Gamma)]_s,\\
 \H^s(\Gamma) &:= [L^2(\Gamma);\H^1(\Gamma)]_s.
\end{split}
\end{align}
\revision{To abbreviate notation, we shall also write $L^2(\Gamma) = H^0(\Gamma)=\H^0(\Gamma)$}.
It follows that all $H^s(\Gamma)$ and $\H^s(\Gamma)$ are dense subspaces of
$L^2(\Gamma)$ with respect to the $L^2(\Gamma)$-norm. Therefore, the 
dual spaces can be understood with respect to the extended $L^2(\Gamma)$-scalar 
product. For $-1\le s<0$, we define
\begin{align}
\begin{split}
 H^{-s}(\Gamma) &:= \H^s(\Gamma)^*,\\
 \H^{-s}(\Gamma) &:= H^s(\Gamma)^*.
\end{split}
\end{align}
It follows that $L^2(\Gamma)$ is dense in $H^{-s}(\Gamma)$ and $\H^{-s}(\Gamma)$
with respect to the associated norms. For $s=0$, we let $\H^0(\Gamma) := L^2(\Gamma) =: H^0(\Gamma)$.

We stress that interpolation theory also states the equalities
\begin{align}
\begin{split}
 H^{-s}(\Gamma) &= [H^{-1}(\Gamma);L^2(\Gamma)]_s,\\
 \H^{-s}(\Gamma) &= [\H^{-1}(\Gamma);L^2(\Gamma)]_s
\end{split}
\end{align}
in the sense of sets and equivalent norms~\cite{mclean}. Moreover, interpolation reveals
the continuous inclusions $\H^{\pm s}(\Gamma)\subseteq H^{\pm s}(\Gamma)$
as well as $\H^{\pm s}(\partial\Omega) = H^{\pm s}(\partial\Omega)$.

The analysis of the hyper-singular integral \revision{equation} further requires
\begin{align}
 H^{\pm s}_\star(\partial\Omega)
 := \set{v \in H^{\revision{\pm s}}(\partial\Omega)}{\dual{v}{1}{L^2(\partial\Omega)}=0}
\end{align}
for $0\le s\le 1$. We define $L^2_\star(\Gamma):=H^0_\star(\Gamma)$.
We again note that interpolation yields the equality
\begin{align}
 H^{\pm s}_\star(\partial\Omega)
 = [L^2_\star(\partial\Omega);H^1_\star(\partial\Omega)]_s.
\end{align}
Finally, \revision{$H^s_0(\Gamma)$} denotes either $\H^s(\Gamma)$ 
\revision{for $\Gamma\subsetneqq\partial\Omega$ resp.\
$H^s_\star(\partial\Omega)$ for $\Gamma=\partial\Omega$. \mf{In either case, $H^s_0(\Gamma)$ contains no constant function different from zero provided
that $\Gamma$ is connected.}}

\subsection{Discrete spaces}
\label{section:discrete}

\noindent
We assume that $\TT_h = \{T_1,\dots,T_N\}$ is a partition of $\Gamma$
into finitely many compact and affine line segments $T\in\TT_h$. With
each element $T\in\TT_h$, we associate an affine bijection
$\gamma_T:[0,1]\to T$. 

For $q\in\N_0$, \revision{let $\PP^q$ denote} the space of polynomials of degree $\le q$ 
on $\R$. With this, we define the space of $\TT_h$-piecewise polynomials by
\begin{align}
 \PP^q(\TT_h) := \set{v_h:\Gamma\to\R}{\forall T\in\TT\quad
 v_h\circ \gamma_T\in\PP^q}.
\end{align}
Note that functions $v_h\in\PP^q(\TT_h)$ are discontinuous in general. 
Special attention is \revision{paid} to the piecewise constants $\PP^0(\TT_h)$.

If continuity is required, we use the space
\begin{align}
 \SS^q(\TT_h) := \PP^q(\TT_h)\cap C(\Gamma)
\end{align}
of continuous splines of piecewise degree $q\ge1$. Special attention is 
\revision{paid} to the Courant space $\SS^1(\TT_h)$ of lowest order.

For the treatment of the hyper-singular integral equation, we additionally
define
\begin{align}
 \widetilde\SS^q(\TT_h) &:= \SS^q(\TT_h)\cap \widetilde H^1(\Gamma),\\
 \SS^q_\star(\TT_h) &:= \SS^q(\TT_h)\cap H^1_\star(\Gamma).
\end{align}
Finally, $\SS^q_0(\TT_h)$ denotes either $\widetilde\SS^q(\TT_h)$ 
\revision{for $\Gamma\subsetneqq\partial\Omega$ resp.\
$\SS^q_\star(\TT_h)$ for $\Gamma=\partial\Omega$.}

\subsection{Projections}
\label{section:operators}

\noindent
Let $X_h$ be a finite dimensional subspace of a Hilbert space $X$.
The $X$-orthogonal projection onto $X_h$ is the unique linear operator
$\P_h:X\to X_h$ such that, for all $x\in X$ and $x_h\in X_h$, it holds
\begin{align}\label{eq:orthproj}
\begin{split}
 \P_h x_h &= x_h,\\
 \dual{\P_hx}{x_h}X &= \dual{x}{x_h}X.
\end{split}
\end{align}
This implies the Pythagoras theorem
\begin{align}
 \norm{x}X^2 = \norm{\P_hx}X^2 + \norm{(1-\P_h)x}X^2
\end{align}
and consequently
\begin{align}\label{eq:bestapproximation}
 \norm{(1-\P_h)x}X
 = \min_{x_h\in X_h}\norm{x-x_h}X.
\end{align}

\def\JJ{\mathcal J}
In~\cite{scottzhang}, a quasi-interpolation operator 
$\JJ_h^\Omega:H^1(\Omega)\to\SS^1(\TT_h^\Omega)$ is introduced. Here, 
$\Omega\subset\R^d$ for $d\ge2$ is a Lipschitz domain, $\TT_h^\Omega$ is a 
conforming triangulation of $\Omega$ into simplices, and $\SS^1(\TT_h^\Omega)$ 
is the lowest-order Courant finite element space. It is shown that $\JJ_h^\Omega$ has a 
local first-order approximation property and is a linear and continuous projection
onto $\SS^1(\TT_h^\Omega)$. Moreover, $\JJ_h^\Omega$ preserves discrete boundary data,
since the boundery values $(\JJ_h v)|_\Gamma$ depend only on the trace 
$v|_\Gamma$ with $\Gamma=\partial\Omega$. 

Let $\TT_h$ denote the partition of
$\Gamma$ induced by $\TT_h^\Omega$. Then, the mentioned properties of $\JJ_h$ 
yield that the restriction 
$\JJ_h:=\JJ_h^\Omega(\cdot)|_\Gamma:H^{1/2}(\Gamma)\to\SS^1(\TT_h)$ to the 
trace space $H^{1/2}(\Gamma)$
yields a well-defined, linear, and continuous projection onto $\SS^1(\TT_h)$
with respect to the $H^{1/2}(\Gamma)$-norm. 
However, arguing along the lines of the domain-based proof 
from~\cite{scottzhang}, we see that $\JJ_h$ has the following properties.
For an element $T\in\TT_h$, we denote by 
\begin{align}\label{eq:hypsing:patch}
 \omega_T:=\bigcup\set{T'\in\TT_h}{T\cap T'\neq\emptyset}
\end{align}
its patch, i.e., the union of $T$ and its (at most two) neighbours. 
We shall use the following properties of $\JJ_h$:

\begin{itemize}
\item[(i)] $\JJ_hv$ is well-defined for all $v\in L^2(\Gamma)$.
\item[(ii)] $(\JJ_hv)|_T$ depends only on the function values $v|_{\omega_T}$ on 
the patch of $T\in\TT_h$.
\item[(iii)] $\JJ_h$ is locally $L^2$-stable, for all $v\in L^2(\Gamma)$,
\begin{align}\label{eq:sz1}
 \norm{(1-\JJ_h)v}{L^2(T)}
 \le \c{sz}\,\norm{v}{L^2(\omega_T)}.
\end{align}
\item[(iv)] $\JJ_h$ is locally $H^1$-stable, for all $v\in H^1(\Gamma)$,
\begin{align}\label{eq:sz2}
 \norm{((1-\JJ_h)v)'}{L^2(T)}
 \le \c{sz}\,\norm{v'}{L^2(\omega_T)}.
\end{align}
\item[(v)] $\JJ_h$ has a first-order approximation property, 
for all $v\in H^1(\Gamma)$,
\begin{align}\label{eq:sz3}
 \norm{(1-\JJ_h)v}{L^2(T)}
 \le \c{sz}\,\norm{hv'}{L^2(\omega_T)}.
\end{align}
\item[(vi)] The constant $\newc{sz}>0$ depends only on the local 
mesh-ratio $\kappa(\TT_h)$.
\end{itemize}
Since $\omega_T$ consists of at most three elements, the $\ell_2$-sums of
the estimates~\eqref{eq:sz1}--\eqref{eq:sz3} also provide global estimates
with $T$ and $\omega_T$ replaced by $\Gamma$. From (iii), we thus see that 
$\JJ_h\in L(L^2(\Gamma);L^2(\Gamma))$. The combination of (iii)--(iv) yields
$\JJ_h\in L(H^1(\Gamma);H^1(\Gamma))$. In particular, the interpolation 
estimate~\eqref{eq:interp:operator} provides $\JJ_h\in L(H^s(\Gamma);H^s(\Gamma))$, for all $0\le s\le1$.

\section{A~posteriori error analysis}
\label{section:aposteriori}

\noindent
In this section, we show that under appropriate assumptions, the ZZ-type error 
estimators proposed provide an upper bound for the error (\emph{reliability}) and, up to some higher-order terms, also a lower
bound for the error (\emph{efficiency}). Our analysis 
builds on
equivalence of seminorms on finite dimensional spaces and scaling arguments.
The elementary, but abstract result employed reads as follows: If $X$ is a 
finite dimensional
space with seminorms $|\cdot|_1$ and $|\cdot|_2$, an estimate of the type
\begin{align}
 |x|_1 \le C\,|x|_2
 \quad\text{for all }x\in X
\end{align}
and some independent constant $C>0$ is equivalent to the \revision{inclusion}
\begin{align}
 \set{x\in X}{|x|_2=0} \subseteq \set{x\in X}{|x|_1=0}
\end{align}
of the respective null spaces. This result is used for polynomial spaces on
\emph{element patches}. To this end, the restricted partition of the patch
$\omega_T$ from~\eqref{eq:hypsing:patch} is denoted by
\begin{align}
 \TT_h|_{\omega_T}:=\set{T'\in\TT_h}{T\cap T'\neq\emptyset}
\end{align}
for all $T\in\TT$.

\subsection{Hyper-singular integral equation}
\label{section:hypsing:analysis}

\noindent
Recall the abbreviate notation $H^{1/2}_0(\Gamma)$ from 
Section~\ref{section:spaces} and note that 
\begin{align}
 \enorm{v}^2:=\dual{Wv}{v}{\revision{L^2}(\Gamma)}
\end{align}
defines an equivalent Hilbert norm on $H^{1/2}_0(\Gamma)$.
Because of $H^{1/2}(\partial\Omega) = \H^{1/2}(\partial\Omega)$
\revision{even with equal norms},
we can simply use the norm $\norm\cdot{\H^{1/2}(\Gamma)}\simeq\enorm\cdot$
throughout the \revision{section}.

We start with the derivation of an upper bound. The proof relies on the 
assumption that $\TT_h$ is the uniform refinement of some coarser mesh
$\TT_{2h}$ and on some saturation assumption~\eqref{eq:saturation}. 
While the first assumption can easily be
achieved implementationally, the latter is essentially equivalent to the
assumption that the numerical scheme has reached an asymptotic regime, 
see~\cite[Section~5.2]{fp} \revision{for discussion and numerical evidence}.

\begin{theorem}\label{prop:hypsing}
Let $\TT_h$ be the uniform refinement of some mesh $\TT_{2h}$, i.e.\
all elements $T\in\TT_{2h}$ are bisected into two sons $T_1,T_2\in\TT_h$ of
half length. Let $u_h\in\SS^1_0(\TT_h)$ and $u_{2h}\in\SS^1_0(\TT_{2h})$ 
be the respective Galerkin solutions. Then, it holds
\begin{align}\label{eq:hypsing:hh2}
 \enorm{u_h-u_{2h}}
 \le\c{hypsing}\,\eta_h
\end{align}
with some constant $\newc{hypsing}>0$ which depends only on
$\Gamma$ and all possible shapes of element patches~\eqref{eq:hypsing:patch}.
Under the saturation assumption
\begin{align}\label{eq:saturation}
 \enorm{u-u_{h}} \le \Csat\,\enorm{u-u_{2h}}
\end{align}
with some uniform constant $0<\Csat<1$, there holds
\begin{align}\label{eq:hypsing:hh22}
 \Csat^{-1}\,\enorm{u-u_h}
 \le \enorm{u-u_{2h}} 
 \le \frac{\c{hypsing}}{(1-\Csat^2)^{1/2}}\,\eta_{h}.
\end{align}
\end{theorem}

\begin{proof}
Let $\Pi_{2h}:L^2(\Gamma)\to\PP^0(\TT_{2h})$ denote the $L^2$-orthogonal 
projection onto the $\TT_{2h}$-piecewise constants, i.e.\ the piecewise 
integral mean operator
\begin{align}\label{eq:L2:P0}
 (\Pi_{2h}\psi)|_{\widehat T}
 = \frac{1}{\length(\widehat T)}
 \int_{\widehat T}\psi\,d\Gamma
 \quad\text{for all }\widehat T\in\TT_{2h}.
\end{align}
According to~\cite{efgp}, it holds that
\begin{align*}
 \enorm{u_h-u_{2h}} \simeq \norm{h^{1/2}(1-\Pi_{2h})u_h'}{L^2(\Gamma)},
\end{align*}
where the hidden constants depend only on $\Gamma$ and the local
mesh-ratio $\kappa(\TT_h)$ from~\eqref{eq:kappa}.
To prove~\eqref{eq:hypsing:hh2}, we will verify 
\begin{align}\label{eq:hypsing:toshow1}
 \norm{h^{1/2}(1-\Pi_{2h})u_{h}'}{L^2(T)}
 \lesssim \norm{h^{1/2}(1-\AA_{h})u_{h}'}{L^2(\omega_T)}
\end{align}
for all $T\in\TT_h$ in the following. Both sides 
of~\eqref{eq:hypsing:toshow1} define seminorms on $\PP^0(\TT_h|_{\omega_T})$, 
where $u_h'$ is replaced by an arbitrary $\psi_h\in\PP^0(\TT_h|_{\omega_T})$. 
It thus suffices to show that $\norm{h^{1/2}(1-\AA_{h})\psi_h}{L^2(\omega_T)}=0$
implies $\norm{h^{1/2}(1-\Pi_{2h})\psi_h}{L^2(T)}$\linebreak$=0$. From 
$\norm{h^{1/2}(1-\AA_{h})\psi_h}{L^2(\omega_T)}=0$ and hence
$\psi_h=\AA_h\psi_h$ on $\omega_T$, we see that $\psi_h$ is constant on
$\omega_T$, since $\psi_h$ is both, $\TT_h$-piecewise constant and continuous
on $\omega_T$. By assumption, $T$ has a brother $T'\in\TT_h$ such that 
$\widehat T = T\cup T'\in\TT_{2h}$. Moreover, the definition of the patch
and $T\cap T'\neq\emptyset$ yield $\widehat T\subseteq \omega_T$. Therefore, 
$\psi_h$ is constant on $\widehat T$ so that $\psi_h = \Pi_{2h}\psi_h$ on 
$\widehat T$. 
This proves $\norm{h^{1/2}(1-\Pi_{2h})\psi_h}{L^2(T)}=0$ and thus verifies
\begin{align*}
 \norm{h^{1/2}(1-\Pi_{2h})\psi_h}{L^2(T)}
 \lesssim \norm{h^{1/2}(1-\AA_{h})\psi_h}{L^2(\omega_T)}
\end{align*}
for all $T\in\TT_h$ and $\psi_h\in\PP^0(\TT_h)$. Finally, a scaling
argument proves that the hidden constant depends only on the shape of the
patch $\omega_T$. We note that each element $T'\in\revision{\TT_h}$ is contained in
at most three patches. Taking the $\ell_2$-sum in~\eqref{eq:hypsing:toshow1}
over all elements $T\in\revision{\TT_h}$, we \revision{arrive} at
\begin{align}\label{eq:tmp}
 \norm{h^{1/2}(1-\Pi_{2h})\psi_h}{L^2(\Gamma)}
 \lesssim \norm{h^{1/2}(1-\AA_{h})\psi_{h}}{L^2(\Gamma)}
\end{align}
for all $\psi_h\in\PP^0(\TT_h)$. Plugging in $\psi_h=u_h'$, 
we conclude the proof of~\eqref{eq:hypsing:hh2}.

The proof of~\eqref{eq:hypsing:hh22} follows from abstract principles.
According to the Galerkin orthogonality
\begin{align*}
 \dual{W(u-u_h)}{v_h}{L^2(\Gamma)} = 0
 \quad\text{for all }v_h\in\SS^1_0(\TT_h),
\end{align*}
we obtain a Pythagoras theorem for the induced Hilbert norm
\begin{align*}
 \enorm{u-u_h}^2 + \enorm{u_h-u_{2h}}^2 = \enorm{u-u_{2h}}^2,
\end{align*}
where we use $v_h = u_h-u_{2h}$. Together with the saturation 
assumption~\eqref{eq:saturation}, this results in
\begin{align*}
 \Csat^{-1}\,\enorm{u-u_h} \le \enorm{u-u_{2h}} 
 \le \frac1{(1-\Csat^2)^{1/2}}\,\enorm{u_h-u_{2h}},
\end{align*}
and~\eqref{eq:hypsing:hh22} follows.
\end{proof}

\begin{remark}
With the same techniques as in the proof of Theorem~\ref{prop:hypsing},
one can prove that the ZZ-type error estimator $\eta_h$ is an upper
bound for the estimator $\mu_h$ from~\cite{cp:hypsing} which is based on
averaging on large patches. The analysis then requires that $\TT_h$ is 
a refinement of a coarser mesh $\TT_{kh}$ for some $k\ge2$ which depends
only on $\Gamma$. Then, the saturation assumption~\eqref{eq:saturation}
is formally avoided. However, the parameter $k$ is still unknown,
although $k=2$ empirically appears to be sufficient, see e.g.\ the
numerical experiments in~\cite{cp:hypsing}.
Moreover, the upper bound~\eqref{eq:hypsing:hh22} holds only up to some 
additional best approximation error
\begin{align*}	
 \enorm{u-u_h} \lesssim \eta_h + \min_{U_h\in\SS^2_0(\TT_{kh})}\enorm{u-U_h}
\end{align*}
with higher-order elements $\SS^2_0(\TT_{kh}):=\PP^2(\TT_{kh})\cap\H^{1/2}(\Gamma)
\subseteq H^1(\Omega)$ which are piecewise quadratic and globally continuous.
If the exact solution $u$ is smooth or if 
the mesh is appropriately graded to the singularities of $u$, this 
additional term is of higher-order. The reader is referred 
to~\cite{cp:steinbach} for further discussions.\qed
\end{remark}

We next prove the lower bound. Unlike the reliability estimate~\eqref{eq:hypsing:hh22},
the following efficiency estimate~\eqref{eq:hypsing:efficient} does not 
rely on the saturation assumption~\eqref{eq:saturation}, but holds only up 
to some further best approximation error with higher-order elements. 
If the exact solution solution $u$ of~\eqref{eq:hypsing} is smooth or if 
the mesh is properly adapted to the singularities of $u$, this term becomes 
a higher-order term.

Let $\SS^{2,1}(\TT_h) := \PP^2(\TT_h)\cap C^1(\Gamma)$ denote the set of all
$\TT_h$-piecewise quadratic polynomials $p$ such that $p$ as well as its 
derivative $p'$ are continuous. With $\SS^{2,1}_0(\TT_h):=\SS^{2,1}(\TT_h)\cap H^{1/2}_0(\Gamma)$, our efficiency result then reads as follows:

\begin{theorem}\label{prop:hypsing:efficient}
It holds
\begin{align}\label{eq:hypsing:efficient}
 \c{hypsing2}^{-1}\,\eta_h \le \enorm{u-u_h} 
 + \min_{U_h\in\SS^{2,1}_0(\TT_h)}\,\enorm{u-U_h}.
\end{align}
The constant $\newc{hypsing2}>0$ depends only on 
$\Gamma$ and all possible shapes of element patches~\eqref{eq:hypsing:patch}.\end{theorem}

The proof requires the following probably well-known lemma. For the 
convenience of the reader, we include the proof also here.

\begin{lemma}\label{lemma:cs}
For $0\le s\le1$, the arc-length derivative induces linear and continuous operators $(\cdot)':H^{s}(\Gamma)\to H^{s-1}(\Gamma)$ and 
$(\cdot)':\H^{s}(\Gamma)\to \H^{s-1}(\Gamma)$.
\end{lemma}

\begin{proof}
For $s=1$, it holds 
\begin{align*}
 \norm{v'}{L^2(\Gamma)}\le \norm{v}{H^1(\Gamma)}
 \quad\text{for all }v\in H^1(\Gamma)
\end{align*}
and, by integration by parts,
\begin{align*}
 \dual{v'}{w}{L^2(\Gamma)} = -\dual{v}{w'}{L^2(\Gamma)}
 \le \norm{v}{L^2(\Gamma)}\norm{w}{H^1(\Gamma)}.
\end{align*}
for all $w\in\H^1(\Gamma)$. Note that here we require either that
$\Gamma=\partial\Omega$ or that $w$ (or $v$) vanishes at the tips of $\Gamma$.
By definition of the duality $H^{-1}(\Gamma)=\H^1(\Gamma)^*$, this yields
\begin{align*}
 \norm{v'}{H^{-1}(\Gamma)}\le \norm{v}{L^2(\Gamma)}
 \quad\text{for all }v\in H^1(\Gamma).
\end{align*}
Since $H^1(\Gamma)$ is dense in $L^2(\Gamma)$, we obtain
continuity of $(\cdot)':L^2(\Gamma)\to H^{-1}(\Gamma)$, i.e.\ the last
estimate holds even for all $v\in L^2(\Gamma)$. Finally, the interpolation 
\revision{estimate~\eqref{eq:interp:operator}} reveals
\begin{align*}
 \norm{v'}{H^{s-1}(\Gamma)} \le \norm{v}{H^s(\Gamma)}
 \quad\text{for all }v\in H^s(\Gamma),
\end{align*}
i.e.\ $(\cdot)':H^{s}(\Gamma)\to H^{s-1}(\Gamma)$ is a linear and
continuous operator, even with operator norm $1$.

To prove the same statement for $(\cdot)':\H^{s}(\Gamma)\to \H^{s-1}(\Gamma)$,
recall the duality $\H^{-1}(\Gamma)=H^1(\Gamma)^*$. With $v\in\H^1(\Gamma)$ and
$w\in H^1(\Gamma)$ all foregoing steps remain valid with nothing but the
obvious modifications.
\end{proof}

\begin{proof}[Proof of Theorem~\ref{prop:hypsing:efficient}]
Let $J_h:L^2(\Gamma)\to\SS^1(\Gamma)$ denote the Scott-Zhang projection
\revision{from Section~\ref{section:operators}}.
We first show that
\begin{align}\label{eq:hypsing:tmp}
 \norm{h^{1/2}(1-\AA_h)\psi_h}{L^2(\Gamma)}
 \lesssim \norm{h^{1/2}(1-J_h)\psi_h}{L^2(\Gamma)}
\end{align}
for all $\psi_h\in\PP^0(\TT_h)$.
To that end, we use a seminorm argument on $\PP^0(\TT_h|_{\omega_T})$:
From $\norm{h^{1/2}(1-J_h)\psi_h}{L^2(\omega_T)}=0$, it follows that
$\psi_h$ is constant on $\omega_T$. By 
definition~\revision{\eqref{eq2:hypsing:clement}}
of $\AA_h$ this yields $\AA_h\psi_h=\psi_h$ on 
$T$. Therefore, we see $\norm{h^{1/2}(1-\AA_h)\psi_h}{L^2(T)}=0$,
and
\begin{align*}
 \norm{h^{1/2}(1-\AA_h)\psi_h}{L^2(T)}
 \lesssim \norm{h^{1/2}(1-J_h)\psi_h}{L^2(\omega_T)} 
\end{align*}
follows. A scaling argument proves that the hidden constant depends only 
on the shape of the patch $\omega_T$. Taking the $\ell_2$-sum 
of the last estimate over all elements 
$T\in\TT_h$, we obtain~\eqref{eq:hypsing:tmp}.

\next Second, we show that
\begin{align}\label{eq:hypsing:tmp2}
 \norm{h^{1/2}(1-J_h)\psi_h}{L^2(\Gamma)}
 \lesssim \norm{\psi_h-\Psi_h}{\H^{-1/2}(\Gamma)}
\end{align}
for all $\psi_h\in\PP^0(\TT_h)$ and $\Psi_h\in\SS^1(\TT_h)$.
Since the Scott-Zhang projection is stable with respect to the 
$h^{1/2}$-weighted $L^2$-norm, \revision{see Section~\ref{section:operators}}, the projection property of $J_h$ gives
\begin{align*}
 \norm{h^{1/2}(1-J_h)\psi_h}{L^2(\Gamma)}
 &= \norm{h^{1/2}(1-J_h)(\psi_h-\Psi_h)}{L^2(\Gamma)}\\
 &\lesssim \norm{h^{1/2}(\psi_h-\Psi_h)}{L^2(\Gamma)}.
\end{align*}
The inverse estimate of~\cite[Thm~3.6]{ghs}
then concludes the proof of~\eqref{eq:hypsing:tmp2}.

\next Finally, let $\P_h:H^{1/2}_0(\Gamma)\to\SS^{2,1}_0(\TT_h)$ denote the 
$H^{1/2}_0(\Gamma)$-orthogonal projection onto $\SS^{2,1}_0(\TT_h)$
with respect to the energy norm $\enorm\cdot$. 
Combining norm equivalence $\enorm\cdot\simeq\norm\cdot{\H^{1/2}(\Gamma)}$
with the estimates~\eqref{eq:hypsing:tmp} and
\eqref{eq:hypsing:tmp2} 
for $\psi_h=u_h'$ and $\Psi_h = (\P_hu_h)'$, we obtain
\begin{align*}
 &\norm{h^{1/2}(1-\AA_h)u_h'}{L^2(\Gamma)}
 \lesssim \norm{(u_h-\P_hu_h)'}{\H^{-1/2}(\Gamma)}\\
 &\quad
\lesssim \norm{(1-\P_h)u_h}{\H^{1/2}(\Gamma)}
 \simeq \enorm{(1-\P_h)u_h}.
\end{align*}
The triangle inequality and stability of $\P_h$ yield
\begin{align*}
 \enorm{(1-\P_h)u_h}
 \le \enorm{(1-\P_h)u} + \enorm{u-u_h}.
\end{align*}
Since $\P_h u$ is the best approximation~\revision{\eqref{eq:bestapproximation}}
of $u$ in $\SS^{2,1}_0(\TT_h)$ with respect to $\enorm\cdot$,
this proves~\eqref{eq:hypsing:efficient}.
\end{proof}

\subsection{Weakly-singular integral equation}
\label{section:weaksing:analysis}

\noindent
We stress that the same results hold as for the hyper-singular 
integral equation. By 
\begin{align}
 \enorm{\revision{w}}^2:=\dual{Vw}{w}{L^2(\Gamma)},
\end{align}
we now denote the Hilbert norm which is induced by the weakly-singular integral operator,
and note that 
$\enorm\cdot\simeq\norm\cdot{\H^{-1/2}(\Gamma)}$ is an equivalent
norm on $\H^{-1/2}(\Gamma)$. The reliability result reads as follows:

\begin{theorem}\label{prop:weaksing}
Let $\TT_h$ be the uniform refinement of some mesh $\TT_{2h}$, i.e.\
all elements $T\in\TT_{2h}$ are bisected into two sons $T_1,T_2\in\TT_h$ of
half length. Let $\phi_h\in\PP^0(\TT_h)$ and $\phi_{2h}\in\PP^0(\TT_{2h})$ 
be the respective Galerkin solutions. Then, it holds
\begin{align}\label{eq:weaksing:hh2}
 \enorm{\phi_h-\phi_{2h}}
 \le\c{weaksing}\,\eta_h
\end{align}
with some constant $\newc{weaksing}>0$ which depends only on
$\Gamma$ and all possible shapes of element patches~\eqref{eq:hypsing:patch}.
Under the saturation assumption
\begin{align}\label{eq:weaksing:saturation}
 \enorm{\phi-\phi_{h}} \le \Csat\,\enorm{\phi-\phi_{2h}}
\end{align}
with some uniform constant $0<\Csat<1$, there holds
\begin{align}\label{eq:weaksing:hh22}
 \Csat^{-1}\,\enorm{\phi-\phi_h}
 \le \enorm{\revision{\phi}-\phi_{2h}} 
 \le \frac{\c{weaksing}}{(1-\Csat^2)^{1/2}}\,\eta_{h}.
\end{align}
\end{theorem}

\begin{remark}
We refer to~\cite{affkp},
where the saturation assumption~\revision{\eqref{eq:weaksing:saturation}} is proved in the frame of the weakly-singular
integral \revision{equation for the Dirichlet problem~\eqref{eq:weaksing:pde}} and $\TT_{2h}$ replaced by some coarser mesh $\TT_{kh}$
with $k\ge2$ depending only on $\Gamma$.\qed
\end{remark}

\begin{proof}[Proof of Theorem~\ref{prop:weaksing}]
We adopt the notation from the proof of Theorem~\ref{prop:hypsing}.
According to~\cite{effp}, it holds that
\begin{align*}
 \enorm{\revision{\phi_h}-\phi_{2h}}\simeq \norm{h^{1/2}(1-\Pi_{2h})\phi_h}{L^2(\Gamma)},
\end{align*}
where the hidden constants depend only on $\Gamma$ and the local mesh-ratio
$\kappa(\TT_h)$ from~\eqref{eq:kappa}. Recall that the operator $\AA_h$ is
now slightly different to the case of the hyper-singular integral equation.
However, the same arguments as in the proof of Theorem~\ref{prop:hypsing}
show that~\eqref{eq:tmp} remains valid. As before the hidden constant involved 
depends on all possible shapes of element patches in $\TT_h$. This 
yields~\eqref{eq:weaksing:hh2}, and~\eqref{eq:weaksing:hh22} follows as before.
\end{proof}

We next prove the lower bound. 
As before, the following efficiency estimate~\eqref{eq:weaksing:efficient} does
not rely on the saturation assumption~\eqref{eq:weaksing:saturation},
but holds only up to some further best approximation error with higher-order elements. 

\begin{theorem}\label{prop:weaksing:efficient}
It holds
\begin{align}\label{eq:weaksing:efficient}
 \c{weaksing2}^{-1}\,\eta_h \le \enorm{\phi-\phi_h} 
 + \min_{\Phi_h\in\SS^1(\TT_h)}\,\enorm{\phi-\Phi_h}.
\end{align}
The constant $\newc{weaksing2}>0$ depends only on $\Gamma$ and all possible 
shapes of element patches~\eqref{eq:hypsing:patch}.
\end{theorem}

\begin{proof}
Arguing along the lines of the proof of Theorem~\ref{prop:hypsing:efficient},
we see that 
\begin{align*}
 \norm{h^{1/2}(1-\AA_h)\phi_h}{L^2(\Gamma)}
 \lesssim \norm{\phi_h-\Psi_h}{\H^{-1/2}(\Gamma)}
\end{align*}
for all $\Psi_h\in\SS^1(\TT_h)$.
Let $\P_h:\H^{-1/2}(\Gamma)\to\SS^1(\TT_h)$ be the orthogonal projection
onto $\SS^1(\TT_h)$ with respect to the energy norm $\enorm\cdot$.
With norm equivalence 
$\enorm\cdot\simeq\norm\cdot{\H^{-1/2}(\Gamma)}$ and the triangle 
inequality, we see for $\Psi_h=\P_h\phi_h$
\begin{align*}
 \enorm{\phi_h-\Psi_h}
 &=\enorm{(1-\P_h)\phi_h}\\
 &\le \enorm{(1-\P_h)\phi} +  \enorm{(1-\P_h)(\phi-\phi_h)} 
 \\
 &\le \enorm{(1-\P_h)\phi} +  \enorm{\phi-\phi_h}.
\end{align*}
Since $\P_h\phi$ is the best approximation of $\phi$ in $\SS^1(\TT_h)$
with respect to $\enorm\cdot$, we conclude the proof.
\end{proof}

\section{Adaptive mesh-refinement}
\label{section:convergence}

\noindent
In this section, we prove that the constants in the a~posteriori estimates
of Section~\ref{section:aposteriori} are uniformly bounded and that the
adaptive algorithms of Section~\ref{section:hypsing:algorithm} and
Section~\ref{section:weaksing:algorithm} are convergent.

\subsection{Notation}
\label{section:notation}

\noindent
For the following analysis, we slightly change the notation for the discrete
quantities. Let $\TT_0$ be the given initial partition of $\Gamma$, the 
adaptive algorithm is started with. Let $\ell=0,1,2,\dots$ denote the counter
for the adaptive loop, i.e.\ we start with $\ell=0$, and $\ell\mapsto\ell+1$ 
is increased in step~(v) of the adaptive algorithm. 

The mesh in the $\ell$-th step of the adaptive loop is denoted by $\TT_\ell$.
With $\TT_\ell$, we associate the local mesh-size $h_\ell\in L^\infty(\Gamma)$ defined in~\eqref{eq:h}. Moreover, $u_\ell\in\SS^1_0(\TT_\ell)$ resp.\ 
$\phi_\ell\in\PP^0(\TT_\ell)$ are the corresponding discrete solutions with 
respective ZZ-type error estimators $\eta_\ell$. 

\def\level{{\tt level}}
\def\T{\mathbb{T}}
Throughout, we assume that mesh-refinement is based on bisection only, i.e.\ 
refined elements are bisected into two sons of half length.
In step~(iv) of the adaptive algorithm, we ensure
\begin{align}\label{eq:level}
 \kappa(\TT_\ell) \le 2\,\kappa(\TT_0)
\end{align}
Algorithmically, this mesh-refinement is stated and analyzed in~\cite{affkp}.
In addition to~\eqref{eq:level}, the 
properties of the mesh-refinement necessary in current proofs of quasi-optimal
convergence rates for adaptive boundary element methods~\cite{fkmp,gantumur}
and adaptive finite element methods~\cite{ckns,s07,s08} are satisfied,
i.e.\ the so-called \emph{overlay estimate} and \emph{mesh-closure estimate}
are valid.
Moreover, bisection and boundedness~\eqref{eq:level} of the local mesh-ratio
guarantee that only a finite 
number of shapes of element patches~\eqref{eq:hypsing:patch} can occur. 
Therefore, the constants in
the a~posteriori analysis of Section~\ref{section:aposteriori} are 
uniformly bounded.

\subsection{Hyper-singular integral equation}
\label{section:hypsing:convergence}

\noindent
The proof of the following theorem follows the concept of
\emph{estimator reduction} proposed in~\cite{afp} for $(h-h/2)$-type
error estimators. \revision{We show} that the ZZ-type error estimator is 
contractive up to some vanishing perturbation
\begin{align}\label{eq:estconv}
 \eta_{\ell+1} \le q\,\eta_\ell + \alpha_\ell
 \text{ with }\revision{0\le\alpha_\ell\xrightarrow{\ell\to\infty}0}
\end{align}
for some $\ell$-independent constant $0<q<1$. In the current frame,
however, the proof that the perturbation $\alpha_\ell$ \revision{tends
to zero,} is much more involved than in~\cite{afp}, 
\revision{since it does not only rely on the a~priori convergence of
Lemma~\ref{lemma:apriori}, but also on a pointwise convergence property
of the averaging operator $\AA_h$.}

\begin{theorem}\label{prop:hypsing:convergence}
Let $(u_\ell)_{\ell\in\N}$ and $(\eta_\ell)_{\ell\in\N}$ be the sequences
of discrete solutions and error estimators generated by the adaptive
algorithm. Then, it holds estimator convergence
\begin{align}\label{eq:hypsing:estconv}
 \lim_{\ell\to\infty}\eta_\ell = 0.
\end{align}
Provided that $\enorm{u-u_\ell}\lesssim \eta_\ell$, 
cf.~Theorem~\ref{prop:hypsing},  we may thus conclude
$\lim\limits_{\ell\to\infty}u_\ell = u$.
\end{theorem}

The proof requires the following lemmas. The first is already found
in the early work~\cite{bv84} and will be applied for 
$H=H^{1/2}_0(\Gamma)$ and $X_\ell = \SS^1_0(\TT_\ell)$ for the
hyper-singular integral equation as well as for
$H=\H^{-1/2}(\Gamma)$ and $X_\ell = \PP^0(\TT_\ell)$ for the
weakly-singular integral equation.

\begin{lemma}[A~priori convergence of Galerkin solutions]
\label{lemma:apriori}
Suppose that $H$ is a Hilbert space and $(X_\ell)_{\ell\in\N}$ is a sequence
of discrete subspaces with $X_\ell\subseteq X_{\ell+1}$. For $u\in H$ and 
$\ell\in\N$, let $u_\ell\in X_\ell$ be the best approximation of $u$. Then,
there exists a limit $u_\infty\in H$ such that $\lim\limits_{\ell\to\infty}
\norm{u_\infty-u_\ell}X=0$.\qed
\end{lemma}

The following lemma recalls local $L^2$-stability and first-order approximation
property of the averaging operator $\AA_\ell$ used.

\begin{lemma}
Let $T\in\TT_\ell$. Then, the operators 
$\AA_\ell:L^2(\Gamma)\to L^2(\Gamma)$ defined 
in~\eqref{eq2:hypsing:clement} resp.~\eqref{eq2:weaksing:clement} are 
locally $L^2$-stable
\begin{align}\label{eq:A:L2}
 \norm{\AA_\ell v}{L^2(T)}
 \le \c{A:L2}\,\norm{v}{L^2(\omega_T)},
\end{align}
for all $v\in L^2(\Gamma)$,
are local $H^1$-stable
\begin{align}\label{eq:A:H1:stab}
 \norm{(\AA_\ell v)'}{L^2(T)}
 \le \c{A:L2}\,\norm{v}{H^1(\omega_T)},
\end{align}
for all $v\in H^1(\Gamma)$,
and have a local first-order approximation property
\begin{align}\label{eq:A:H1}
 \norm{(1-\AA_\ell)v}{L^2(T)}
 \le \c{A:L2}\,\norm{h_\ell v'}{L^2(\omega_T)},
\end{align}
for all $v\in H^1(\Gamma)$. Here, $\omega_T$ denotes the element
patch~\eqref{eq:hypsing:patch} of $T\in\TT_\ell$, and $\newc{A:L2}>0$ 
depends only on $\Gamma$ and the mesh-refinement chosen.
\end{lemma}

\begin{proof}
The proof follows as for usual Cl\'ement-type operators in finite element
analysis, cf e.g.~\cite{bs,scottzhang}. Scaling arguments prove that the 
constants involved depend only on the shape of the element patch $\omega_T$. 
The mesh-refinement chosen guarantees that only finitely many pat\-ches occur
so that these constants depend, in fact, only on the boundary $\Gamma$ and the 
mesh-refinement strategy.
\end{proof}

The following proposition is more general than required for the proof of
Theorem~\ref{prop:hypsing:convergence}. However, it might be of general
interest and might have further applications, since it also applies to FEM
and higher dimensions even with the same proof.

\begin{proposition}[A~priori convergence of averaging operators]
\label{lemma:scottzhang}
Given the sequence $(\TT_\ell)_{\ell\in\N}$ of adaptively generated meshes,
let $\AA_\ell:L^2(\Gamma)\to H^1(\Gamma)$ be a linear operator which 
satisfies~\eqref{eq:A:L2}--\eqref{eq:A:H1}. Assume that, for all elements
$T\in\TT_\ell$
and all functions $v\in L^2(\Gamma)$, $(\AA_\ell v)|_T$ depends only on 
the function values $v|_{\omega_T}$ on the element patch~\eqref{eq:hypsing:patch}. Then, there a exists a limit
operator $\AA_\infty:L^2(\Gamma)\to L^2(\Gamma)$ which satisfies the 
following:
\begin{itemize}
\item[\rm(i)] For all $0\le s\le1$, $\AA_\infty:H^s(\Gamma)\to H^s(\Gamma)$ is a
well-defined linear and continuous operator.
\item[\rm(ii)] For all $0\le s<1$, $\AA_\infty$ is the pointwise limit of 
$\AA_\ell$, i.e., for all $v\in H^s(\Gamma)$ it holds
\begin{align}
 \lim_{\ell\to\infty}\norm{(\AA_\infty-\AA_\ell)v}{H^s(\Gamma)}=0.
\end{align}
\item[\rm(iii)] For all $v\in H^1(\Gamma)$, $\AA_\ell v$ converges weakly in 
$H^1(\Gamma)$ towards $\AA_\infty v$ as $\ell\to\infty$.
\end{itemize}
\end{proposition}

\begin{proof}
For the proof, let $\omega_\ell(\gamma):=\bigcup\set{T\in\TT_\ell}{T\cap\overline\gamma\neq\emptyset}$ denote the patch of subsets $\gamma\subseteq\Gamma$ with respect to $\TT_\ell$. We follow the ideas from~\cite{msv} and define the following subsets
of $\Gamma$:
\mf{\begin{align*}
 \Gamma_\ell^0 &:= \bigcup\set{T\in\TT_\ell}{\omega_\ell(T)\subseteq\bigcup\big(\bigcap_{j=\ell}^\infty\TT_j\big)},\\
 \Gamma_\ell &:= \bigcup\big\{T\in\TT_\ell\,:\,\text{Exists }k\ge0\text{ s.t.\ }
 \omega_T \text{ is at least }
 \\&\qquad\qquad\qquad\qquad
 \text{uniformly refined in }\TT_{\ell+k}\big\},\\
 \Gamma_\ell^* &:= \Gamma \backslash (\Gamma_\ell\cup\Gamma_\ell^0).
\end{align*}}
According to~\cite[Corollary~4.1]{msv}, it holds that 
\begin{align}\label{eq:conv:h}
 \norm{h_\ell}{L^\infty(\omega_\ell(\Gamma_\ell))}
 \simeq \norm{h_\ell}{L^\infty(\Gamma_\ell)} \xrightarrow{\ell\to\infty}0.
\end{align}
Let $v\in L^2(\Gamma)$ and $\eps>0$ be arbitrary. Since $H^1(\Omega)$ is dense in $L^2(\Gamma)$,
we find $v_\eps\in H^1(\Gamma)$ such that $\norm{v-v_\eps}{L^2(\Gamma)}\le\eps$.
Due to the local $L^2$-stability~\eqref{eq:A:L2} and the approximation property~\eqref{eq:A:H1} of $\AA_\ell$, 
we obtain
\begin{align*}
 \norm{(1-\AA_\ell)v}{L^2(\Gamma_\ell)}
 &\lesssim \norm{(1-\AA_\ell)v_\eps}{L^2(\Gamma_\ell)}
 + \eps
 \\&\lesssim \norm{h_\ell\nabla v_\eps}{L^2(\omega_\ell(\Gamma_\ell))}+\eps.
\end{align*}
According to~\eqref{eq:conv:h}, we find $\ell_0\in\N$ such that
$$\norm{h_\ell\nabla v_\eps}{L^2(\omega_\ell(\Gamma_\ell))}\le
\norm{h_\ell}{L^\infty(\omega_\ell(\Gamma_\ell))}
\norm{\nabla v_\eps}{L^2(\Gamma)}\le \eps$$ for all $\ell\ge\ell_0$.
This proves
\begin{align}\label{eq:conv:ell}
 \norm{(1-\AA_\ell)v}{L^2(\Gamma_\ell)} \lesssim \eps
 \quad\text{for }\ell\ge\ell_0.
\end{align}
\cite[Proposition~4.2]{msv} states $|\Gamma_\ell^*|\to0$ as $\ell\to\infty$.
Due to the non-concentration of Lebesgue functions, this yields
\begin{align}\label{eq:dp1}
 \norm{v}{L^2(\omega_\ell(\Gamma_\ell^\star))}
 \le \eps
 \quad\text{for some }\ell_1\in\N\text{ and all }\ell\ge\ell_1.
\end{align}
Let $\ell\ge\max\{\ell_0,\ell_1\}$ and $k\ge0$. For $T\in\TT_\ell$, 
the definition of $(\AA_\ell v)|_T$ depends only on $v|_{\omega_\ell(T)}$. By definition of $\Gamma_\ell^0$, 
we obtain
\begin{align*}
 \norm{(\AA_\ell-\AA_{\ell+k})v}{L^2(\Gamma_\ell^0)}=0.
\end{align*}
With local $L^2$-stability~\eqref{eq:A:L2} and~\eqref{eq:dp1}, we see
\begin{align*}
 \norm{(\AA_\ell - \AA_{\ell+k})v}{L^2(\Gamma_\ell^*)}
 &\lesssim \norm{v}{L^2(\omega_\ell(\Gamma_\ell^*))} + 
 \norm{v}{L^2(\omega_{\ell+k}(\Gamma_\ell^*))} 
 \\&\le 2\,\norm{v}{L^2(\omega_\ell(\Gamma_\ell^*))}
 \lesssim\eps.
\end{align*}
Moreover,~\eqref{eq:conv:ell} and a triangle inequality prove
\begin{align*}
 \norm{(\AA_\ell-\AA_{\ell+k})v}{L^2(\Gamma_\ell)}
 \lesssim \eps.
\end{align*}
The combination of the last three estimates yields
\begin{align*}
 \norm{(\AA_\ell-\AA_{\ell+k})v}{L^2(\Gamma)} \lesssim \eps.
\end{align*}
Altogether, $(\AA_\ell v)_\ell$ is thus a Cauchy sequence in $L^2(\Gamma)$ and
hence convergent to some limit 
$\AA_\infty v:=\lim_\ell \AA_\ell v \in L^2(\Gamma)$.
Elementary calculus predicts that this provides a well-defined linear operator
$\AA_\infty:L^2(\Gamma)\to L^2(\Gamma)$, and the Banach-Steinhaus theorem even
predicts continuity $\AA_\infty\in L(L^2(\Gamma); L^2(\Gamma))$.

Second, the $H^1$-stability~\eqref{eq:A:H1:stab} yields that $\AA_\ell\in L^2(H^1(\Gamma);\linebreak H^1(\Gamma))$ are uniformly continuous operators. For 
$v\in H^1(\Gamma)$, the sequence $(\AA_\ell v)_\ell$ is hence bounded in $H^1(\Gamma)$ and thus
admits a weakly convergent subsequence $\AA_{\ell_k}v\rightharpoonup w$ weakly
in $H^1(\Gamma)$ as $k\to\infty$. The Rellich compactness theorem yields $\AA_{\ell_k}v\to w$ strongly in $L^2(\Omega)$. Uniqueness of limits therefore reveals
$\AA_\infty v = w \in H^1(\Gamma)$. Iterating this argument, we see that each 
subsequence of $\AA_\ell v$ admits a further subsequence such that $\AA_{\ell_{k_j}}v$ converges to $\AA_\infty v\in H^1(\Gamma)$ weakly in 
$H^1(\Gamma)$. By elementary calculus, this implies weak convergence 
$\AA_\ell v\rightharpoonup \AA_\infty v$ in $H^1(\Gamma)$ for the entire 
sequence. Again, the Banach-Steinhaus theorem applies and proves that
$\AA_\infty \in L(H^1(\Gamma);H^1(\Gamma))$.

Third, the remaining claims follow from interpolation. The interpolation
estimate~\eqref{eq:interp:operator} implies that the operator 
$\AA_\infty \in L(H^s(\Gamma);H^s(\Gamma))$ is 
well-defined, linear, and continuous. Moreover, the 
estimate~\eqref{eq:interp:norm} of the interpolation norm and boundedness
of weakly convergent sequences yields
\begin{align*}
 &\norm{(\AA_\infty-\AA_\ell)v}{H^s(\Gamma)}\\
 &\quad\le \norm{(\AA_\infty-\AA_\ell)v}{L^2(\Gamma)}^{1-s}
 \norm{(\AA_\infty-\AA_\ell)v}{H^1(\Gamma)}^s
 \xrightarrow{\ell\to0}0
\end{align*}
for all $0<s<1$ and $v\in H^1(\Gamma)$. By density of $H^1(\Gamma)$ 
in $H^s(\Gamma)$ and stability of $\AA_\ell$, this results in
pointwise convergence $\norm{(\AA_\infty-\AA_\ell)v}{H^s(\Gamma)}\to0$ for all $v\in H^s(\Gamma)$.
\end{proof}

\begin{proof}[Proof of Theorem~\ref{prop:hypsing:convergence}]
The triangle inequality shows
\begin{align}\label{eq1:hypsing:estconv}
\begin{split}
 \revision{\eta_{\ell+1}}
 \le& \norm{h_{\ell+1}^{1/2}(1-\AA_\ell)u_\ell'}{L^2(\Gamma)}
 \\&+ \norm{h_{\ell+1}^{1/2}(1-\AA_{\ell+1})(u_{\ell+1}-u_\ell)'}{L^2(\Gamma)}
 \\&+ \norm{h_{\ell+1}^{1/2}(\AA_{\ell+1}-\AA_\ell)u_\ell'}{L^2(\Gamma)}.
\end{split}
\end{align}
\next For the first term, we argue analogously to~\cite{afp}: According to bisection,
we have $h_{\ell+1}|_T = \frac12\,h_\ell|_T$ for refined elements
$T\in\TT_\ell\backslash\TT_{\ell+1}$. This gives
\begin{align*}
 &\norm{h_{\ell+1}^{1/2}(1-\AA_\ell)u_\ell'}{L^2(\Gamma)}^2
 \\&\qquad
 \le \sum_{T\in\TT_\ell\cap\TT_{\ell+1}}\eta_\ell(T)^2
 + \frac12\sum_{T\in\TT_\ell\backslash\TT_{\ell+1}}\eta_\ell(T)^2
 \\&\qquad
 = \eta_\ell^2 - \frac12\sum_{T\in\TT_\ell\backslash\TT_{\ell+1}}\eta_\ell(T)^2.
\end{align*}
Since at least all marked elements are refined, 
the D\"orfler marking strategy~\eqref{eq:doerfler} in step~(iii)
of the adaptive algorithm yields
\begin{align*}
 \sum_{T\in\TT_\ell\backslash\TT_{\ell+1}}\eta_\ell(T)^2
 \ge \sum_{T\in\MM_\ell}\eta_\ell(T)^2
 \ge \theta\,\eta_\ell^2.
\end{align*}
Combining the last two estimates, we see
\begin{align}\label{eq:hypsing:step1}
 \norm{h_{\ell+1}^{1/2}(1-\AA_\ell)u_\ell'}{L^2(\Gamma)}
 \le (1-\theta/2)^{1/2}\,\eta_\ell.
\end{align}
\next Next, we consider the second term in~\eqref{eq1:hypsing:estconv}. 
The local $H^1$-stability~\eqref{eq:A:L2} yields
\begin{align*}
 \norm{h_{\ell+1}^{1/2}(1\!-\!\AA_{\ell+1})(u_{\ell+1}\!-\!u_\ell)'}{L^2(\Gamma)}
 \lesssim \norm{h_{\ell+1}^{1/2}(u_{\ell+1}\!-\!u_\ell)'}{L^2(\Gamma)}.
\end{align*}
The inverse estimate of~\cite[Thm.~3.6]{ghs} gives
\begin{align*}
 &\norm{h_{\ell+1}^{1/2}(u_{\ell+1}-u_\ell)'}{L^2(\Gamma)}
 \lesssim \norm{(u_{\ell+1}-u_\ell)'}{\H^{-1/2}(\Gamma)}
 \\&\qquad
 \lesssim \norm{u_{\ell+1}-u_\ell}{\H^{1/2}(\Gamma)}
 \simeq\enorm{u_{\ell+1}-u_\ell}.
\end{align*}
Together with the a~priori convergence of Lemma~\ref{lemma:apriori},
we thus see
\begin{align}\label{eq:hypsing:step2}
 \norm{(1-\AA_{\ell+1})(u_{\ell+1}-u_\ell)'}{L^2(\Gamma)}
 \xrightarrow{\ell\to\infty}0.
\end{align}
\next Third, we consider the last term in~\eqref{eq1:hypsing:estconv}: 
Let $\eps>0$.
According to the a~priori convergence of Lemma~\ref{lemma:apriori},
there exists an index $k_0\in\N$ such that
\begin{align*}
 \norm{u_\ell-u_k}{\H^{1/2}(\Gamma)}\le \eps
 \quad\text{for all }k,\ell\ge k_0.
\end{align*}
According to the pointwise a~priori convergence of $\AA_\ell$
from Lemma~\ref{lemma:scottzhang}, there
exists an index $\ell_0\in\N$ such that
\begin{align*}
 \norm{(\AA_{\ell+1}-\AA_\ell)u_{k_0}'}{L^2(\Gamma)}\le \eps
 \quad\text{for all }\ell\ge\ell_0.
\end{align*}
Moreover, the local $L^2$-stability~\eqref{eq:A:L2} of the operators yields
\begin{align*}
 \norm{h_{\ell+1}^{1/2}(\AA_{\ell+1}-\AA_\ell)\psi}{L^2(\Gamma)}
 \lesssim \norm{h_{\ell+1}^{1/2}\psi}{L^2(\Gamma)}.
\end{align*}
Plugging in $\psi = (u_\ell-u_{k_0})'$, the usual inverse estimate
from~\cite[Thm.~3.6]{ghs} shows
\begin{align*}
 &\norm{h_{\ell+1}^{1/2}(\AA_{\ell+1}\!-\!\AA_\ell)(u_\ell\!-\!u_{k_0})'}{L^2(\Gamma)}
 \lesssim\norm{h_{\ell+1}^{1/2}(u_\ell-u_{k_0})'}{L^2(\Gamma)}
 \\&\qquad
 \lesssim\norm{(u_\ell-u_{k_0})'}{\H^{-1/2}(\Gamma)}
 \lesssim\norm{u_\ell-u_{k_0}}{\H^{1/2}(\Gamma)},
\end{align*}
where the hidden constants depend only on $\Gamma$ and uniform boundedness
of the local mesh-ratio $\kappa(\TT_\ell)$.
For $\ell\ge\max\{k_0,\ell_0\}$, we thus obtain
\begin{align*}
 &\norm{h_{\ell+1}^{1/2}(\AA_{\ell+1}-\AA_\ell)u_\ell'}{L^2(\Gamma)}
 \\&\qquad
 \lesssim \norm{(\AA_{\ell+1}-\AA_\ell)u_{k_0}'}{L^2(\Gamma)}
 + \norm{u_\ell-u_{k_0}}{\H^{1/2}(\Gamma)}
 \\&\qquad
 \le 2\eps.
\end{align*}
This proves
\begin{align}\label{eq:hypsing:step3}
 &\norm{h_{\ell+1}^{1/2}(\AA_{\ell+1}-\AA_\ell)u_\ell'}{L^2(\Gamma)}
 \xrightarrow{\ell\to\infty}0.
\end{align}
\next 
Altogether,~\eqref{eq:hypsing:step1}--\eqref{eq:hypsing:step3} prove
\begin{align*}
 \eta_{\ell+1} \le (1-\theta/2)^{\revision{1/2}}\,\eta_\ell + \alpha_\ell
 \text{ with }0\le\alpha_\ell\xrightarrow{\ell\to\infty}0.
\end{align*}
Since $0<\theta\le1$, the error estimator is thus contractive up
to a zero sequence. Therefore, elementary calculus concludes~\eqref{eq:hypsing:estconv}.
\end{proof}

\subsection{Weakly-singular integral equation}
\label{section:weaksing:convergence}

\noindent
As for the hyper-singular integral equation, we have the following 
convergence result for the adaptive algorithm of Section~\ref{section:weaksing:algorithm}.

\begin{theorem}\label{prop:weaksing:convergence}
Let $(\phi_\ell)_{\ell\in\N}$ and $(\eta_\ell)_{\ell\in\N}$ be the sequences
of discrete solutions and error estimators generated by the adaptive
algorithm. Then, it holds
\begin{align}
 \lim_{\ell\to\infty}\eta_\ell = 0.
\end{align}
Provided that $\enorm{\phi-\phi_\ell}\lesssim \eta_\ell$, 
cf.~Theorem~\ref{prop:weaksing},   we may thus conclude
$\lim\limits_{\ell\to\infty}\phi_\ell = \phi$.
\end{theorem}

\begin{proof}
The proof follows analogously to that of Theorem~\ref{prop:hypsing:convergence}.
\end{proof}

\bigskip

\noindent
{\bf Acknowledgement.}
The authors MF, TF, and DP acknowledge support through the Austrian Science Fund (FWF) under grant P21732 \emph{Adaptive Boundary Element Method}. 
\revision{MK acknowledges support by CONICYT project Anillo ACT1118 (ANANUM).}


\newcommand{\etalchar}[1]{$^{#1}$}

\end{document}